\documentclass[12pt]{article}
\usepackage{amsmath,amsfonts,amssymb,amsthm}
\usepackage{xcolor}
\usepackage{graphicx}
\usepackage{diagbox}
\usepackage{booktabs}
\usepackage{multirow}
\usepackage[hidelinks,bookmarks,breaklinks=true]{hyperref}

\newtheorem{proposition}{Proposition}[section]
\newtheorem{theorem}[proposition]{Theorem}
\newtheorem{corollary}[proposition]{Corollary}
\newtheorem{lemma}[proposition]{Lemma}
\newtheorem{remark}[proposition]{Remark}

\newtheorem{example}[proposition]{Example}

\newcommand{\nc}{\newcommand}
\nc{\I}{{\mathbf 1}}

\nc{\bN}{{\mathbf N}}
\nc{\bM}{{\mathbf M}}
\nc{\cB}{{\mathcal B}}
\nc{\cL}{{\mathcal L}}
\nc{\R}{{\mathbb R}}
\nc{\C}{{\mathbb C}}
\nc{\N}{{\mathbb N}}
\nc{\Z}{{\mathbb Z}}
\nc{\te}{{\mathrm{e}}}
\nc{\ti}{{\mathrm{i}}}
\nc{\D}{{\text{d}}}

\nc{\BP}{\mathbb{P}}
\nc{\BE}{\mathbb{E}}
\nc{\BQ}{\mathbb{Q}}
\DeclareMathOperator{\BV}{{\mathbb Var}}

\DeclareMathOperator{\Cum}{{\mathbb Cum}}

\numberwithin{equation}{section}

\begin{document}

\renewcommand{\thefootnote}{\fnsymbol{footnote}}
\author{Michael A. Klatt\footnotemark[1], G\"unter Last\footnotemark[2]\, and Norbert Henze\footnotemark[3]}
\footnotetext[1]{michael.klatt@dlr.de;
German Aerospace Center (DLR), Institute for AI Safety and Security, Wilhelm-Runge-Str.~10, 89081 Ulm, Germany;
German Aerospace Center (DLR), Institute of Frontier Materials on Earth and in Space, Functional, Granular, and Composite Materials, 51170 Cologne, Germany;
Department of Physics, LMU M\"unchen, Schellingstr.~4, 80799 Munich, Germany.}
\footnotetext[2]{guenter.last@kit.edu;
Karlsruhe Institute of Technology, Institute for Stochastics, 76131 Karlsruhe, Germany.}
\footnotetext[3]{norbert.henze@kit.edu;
Karlsruhe Institute of Technology, Institute for Stochastics, 76131 Karlsruhe, Germany.}

\title{A genuine test for hyperuniformity}
\date{\today}
\maketitle

\begin{abstract}
\noindent %
We introduce a rigorous and sensitive significance test for
hyperuniformity that yields reliable results even from a single sample.
Our approach is based on a detailed analysis of the empirical Fourier
transform of a stationary point process in $\R^d$. For large system
sizes, we derive the asymptotic covariances and establish a multivariate
central limit theorem (CLT) for these empirical Fourier transforms.
Their absolute square value, the scattering intensity, is then used as
the standard estimator of the  structure factor. The above CLT holds for
a preferably large class of point processes, and whenever this is the
case, the scattering intensity satisfies a multivariate limit theorem as
well. Hence, we can use the likelihood ratio principle to test for
hyperuniformity. Remarkably, the asymptotic distribution of the
resulting test statistic is universal under the null hypothesis of
hyperuniformity. We obtain its explicit form from simulations with very
high accuracy. The novel test precisely keeps a nominal significance
level for hyperuniform models, and it rejects non-hyperuniform examples
with high power even in borderline cases. Moreover, it does so given
only a single sample with a practically relevant system size.
\end{abstract}


\vspace{0.2cm}

\section{Introduction}\label{intro}

Disordered {\em hyperuniform} point patterns exhibit an anomalous
suppression of density fluctuations on large scales~\cite{Torquato2018}.
They can be both isotropic like a liquid and homogeneous like a crystal,
representing a new state of matter. As a result, they have attracted
significant and rapidly growing interest across various disciplines,
including physics~\cite{TorquatoStillinger2003, GJJLPL03, Torquato2018,
  lei_nonequilibrium_2019, klatt_universal_2019, wilken_random_2021,
nizam_dynamic_2021, SkolTorquato23, GaCaBertier23, BackAlSaVoigt24},
biology~\cite{ JiaoEtAl2014, huang_circular_2021}, materials
science~\cite{ chen_stonewales_2021, zheng_disordered_2020,
gorsky_engineered_2019-1}, and probability
theory~\cite{GL17a,GL17b,KLY20,Chatt19,Lachi21, BjoerklundHartnick24,
DeFlHueLe24, LachYogesh24, KLLY25, HuesmannLeble26}, among others.
However, despite this widespread interest, a rigorous statistical  test
to reliably assess the hypothesis of hyperuniformity for a given point
pattern has been lacking. This absence poses a risk of ambiguous
findings.

Hyperuniformity of a stationary point process is a second-order property
and can be defined in either real or Fourier space. In real space, a
point process is said to be hyperuniform if, in the limit of large
window sizes, the variance of the number of points within a window grows
more slowly than the volume of the window. Alternatively, in Fourier
space, hyperuniformity is characterized by the vanishing of the
so-called {\em structure factor} at small wave vectors (see
\cite{hansen_theory_2013} and \eqref{SFgeneral}). The structure factor
is essentially the Fourier transform of the reduced {\em pair
correlation function}, which plays a central role in the theory of point
processes~\cite{CSKM13,LastPenrose17}.

These definitions have led to various ad hoc tests of hyperuniformity
based on heuristic principles, which have been widely used in applied
research. A straightforward approach estimates the variance of the
number of points as a function of the window
size~\cite{dreyfus_diagnosing_2015, nizam_dynamic_2021}. However, this
{\em real space method} requires heuristic rules that avoid a
significant overlap between sampling
windows~\cite{klatt_characterization_2016, torquato_local_2021}. While a
rigorous test of hyperuniformity could be designed using independent
samples for each observation window, such an approach would require
millions of samples~\cite{torquato_local_2021}, making it impractical
for  real-world applications.

An alternative common approach involves extrapolating the structure
factor in the limit of small wave
vectors~\cite{dreyfus_diagnosing_2015,atkinson_critical_2016,
klatt_universal_2019}. Typically, this is done via a least-squares fit,
which  does not enforce the non-negativity constraint of the structure
factor, preventing it from serving as a rigorous significance test. This
drawback is avoided in a recent work~\cite{HGBL22} that provides a
detailed comparison of different estimators of the structure factor and
essentially restricts an empirical test to the smallest accessible
wavenumber. However, similar to real space sampling, this approach also
requires an independent sample for each data point and does not
extrapolate the structure factor to the origin.

In this paper, we address the limitations of previous ad hoc methods by
leveraging detailed distributional properties of the empirical Fourier
transform. Specifically, we focus on point processes whose estimated
structure factor follows a multivariate limit theorem. We hypothesize
that this limit theorem holds for a broad class of point processes, an
assumption supported by Theorem \ref{clt} and by \cite[Theorem
1]{BiscioWaage19}. Additionally, extensive simulations across various
point  processes further reinforce our hypothesis. By working in an
asymptotic setting, we introduce a rigorous {\em likelihood ratio test}
that effectively distinguishes  hyperuniform from  non-hyperuniform
samples for practically relevant models and parameter choices. Our
approach relies on asymptotic approximations, which naturally align with
the study of a long-range property such as hyperuniformity. Notably,
high-precision simulations confirm that these approximations remain
accurate even for a moderate system sizes of around 1\,000 points per
sample.
Here, ``system size'' refers to the mean number of points
per sample. Since we work with a fixed number density (intensity),
it can also equivalently refer to the linear dimension of the observation
window, such as the side length of a cube.

\begin{figure}[t]
  \centering
  \includegraphics[width=\linewidth]{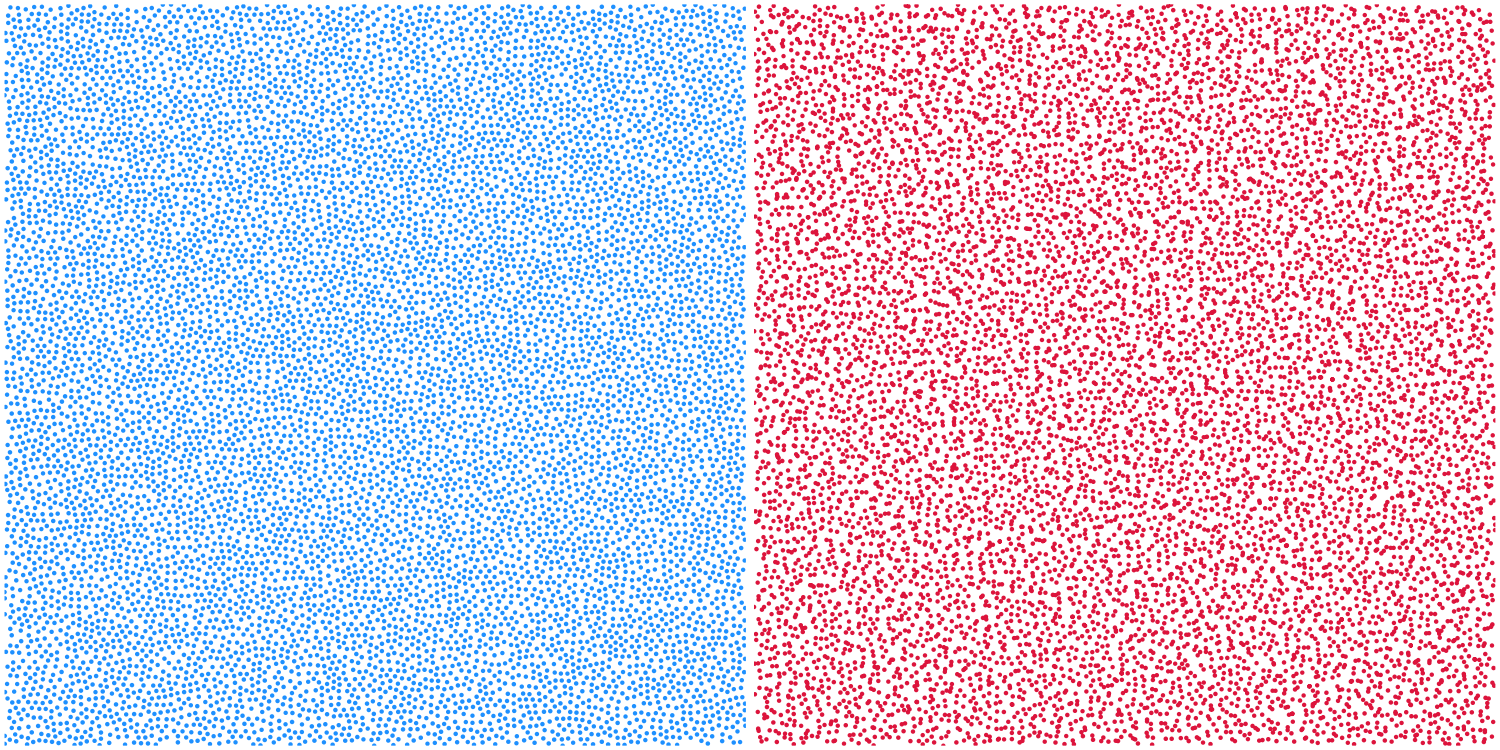}
  \caption{Realizations of a non-hyperuniform (left) and a hyperuniform
  (right) point process, both with about 10000 points.}
  \label{fig:sample-intro}
\end{figure}

To illustrate the difficulty of distinguishing a hyperuniform point
pattern from a non-hyperuniform one by eye,
Figure~\ref{fig:sample-intro} provides an instructive example. The left
panel shows a realization of a (non-hyperuniform) Mat\'ern III process
near saturation~\cite{CSKM13}. Although this process exhibits a high
degree of local order, it possesses Poisson-like long-range density
fluctuations, which are fundamentally incompatible with hyperuniformity.
In contrast, the right panel presents a hyperuniform process. This
sample was generated by first simulating a stealthy hyperuniform point
pattern~\cite{Torquato2018} and then independently perturbing each point
according to a Gaussian distribution. Both processes were simulated on
the flat torus. The subtle difference in long-range order between these
two samples is clearly detected by our novel test statistic. For the
non-hyperuniform sample, the value of the test statistic exceeds 300,
far surpassing the critical value of 2.39, which corresponds to a
nominal significance level of 5\%. In contrast, for the hyperuniform
sample, the test statistic attains the value 0.04, demonstrating the
remarkable sensitivity of the novel test for assessing hyperuniformity.

We exploit three asymptotic approximations in our test procedure. These
approximations are expected to hold simultaneously in the limit of large
system sizes; indeed, our simulations indicate rapid convergence even
for samples containing only a few hundred points. The three
approximations are as follows. First, the previously mentioned
multivariate limit theorem. Second, a putative limit theorem for our
test statistics, obtained from extensive simulations. Third, a parametric
expansion of the structure factor around the origin, which we use to
extrapolate its behavior. Specifically, in the third approximation, we
assume that the structure factor $S$ satisfies
\begin{align}\label{skparabola}
S(k)=s+t\|k\|^\alpha+o(\|k\|^\alpha),\quad \text{as $k\to 0$},
\end{align}
for some fixed $\alpha > 0$ and parameters $s\ge 0$ and $t\in\R$, at
least for wave vectors $k\in\R^d$ in a neighborhood of the origin. In
this setting, the hyperuniform case is characterized by $s=0$. The
approximation thus encompasses, in principle, all three classes of
hyperuniformity as defined, e.g., in \cite{Torquato2018}: class I with
$\alpha >1$, class II with $\alpha =1$, and class III with $0< \alpha <
1$. The third class corresponds to the least uniform point patterns for
which hyperuniformity is most difficult to detect. If there exists an
isotropic pair correlation function $g_2$ satisfying
\begin{align}\label{efinite2}
\int \|x\|^2|g_2(x)-1|\,\D x<\infty,
\end{align}
then assumption \eqref{skparabola} holds with $\alpha=2$, which can
therefore be considered as the most common case. If $|g_2(x)-1|\sim
\|x\|^{-d-\alpha}$ as $\|x\|\to\infty$, for $\alpha\in(0,2)$, then
\eqref{skparabola} holds as well. In the
recent preprint \cite{MastrilliBL25}, the authors assume $s=0$ and
propose methods for estimating the exponent $\alpha$. In contrast, we
fix $\alpha$ (typically $\alpha=2$) and focus on testing the
hyperuniformity hypothesis $s=0$.

In the following, we summarize our results and relate them to the
relevant literature. Section \ref{s2} introduces key concepts from the
theory of stationary point processes and formally defines
hyperuniformity. Following the classical reference \cite{Brillinger72}
(see also the recent survey \cite{HGBL22}), Section \ref{scovariance}
presents the empirical Fourier transform of a point process, defined
using a general taper function. Proposition \ref{pcov} establishes the
asymptotic covariance structure of this transform. While the special
case $d=1$ has been implicitly addressed in \cite{Brillinger72}, a
rigorous result for general dimensions has, to our knowledge, not been
previously established. A related paper is \cite{HeinPro10}, where the
authors examine the bias and variance of kernel estimators of the
asymptotic variance. In Section~\ref{sclt}, we extend a result of
\cite{Brillinger72} to general dimensions by proving a multivariate
central limit theorem (CLT) for the empirical Fourier transform (Theorem
\ref{clt}). As in \cite{Brillinger72}, we rely on an assumption
regarding the reduced factorial cumulant measures. Our proof does also
benefit from the classical work of  \cite{HeinrichSchmidt85}, which
employs a similar hypothesis to establish a related CLT; see also
\cite{HeinrichKlein14} for a more recent contribution to the asymptotic
normality of kernel estimators of correlation functions. We illustrate
our findings using Poisson cluster and $\alpha$-determinantal processes.

In the following, we use the {\em empirical scattering intensity}
\eqref{scattering} as a well-established estimator of the structure
factor; see \cite{Torquato2018}.  Section \ref{asymptsctt} explores
its multivariate asymptotic behavior for different wave vectors and
for large system sizes. We will see that the resulting limit
distribution consists of exponential random variables, which are
independent across different wave vectors (Proposition
\ref{tiidexp}). This crucial result requires the underlying point process
to satisfy the multivariate central limit theorem
\eqref{egoodpp}. Theoretically, this assumption is supported by 
Theorem \ref{clt} and \cite[Theorem 1]{BiscioWaage19}; see Remark
\ref{remgood2}. To validate these theoretical findings, we have
simulated six different models, including two hyperuniform ones, across
the first three dimensions. In each case, the theoretically predicted
behavior is observed over up to six orders of magnitude, even for
small system sizes of about 1\,000 points. These simulations indicate
a fast rate of convergence, ensuring that the asymptotic distribution
is attained with high accuracy at practically relevant system sizes.

In Section \ref{secML}, we shift the focus from point processes to the
asymptotic setting established in Proposition~\ref{tiidexp}.  By
employing the parametrization \eqref{skparabola} of the structure
factor, we obtain a statistical model with only two parameters,
denoted as $s$ and $t$.  This additional approximation and
simplification can be justified by taking sufficiently large system
sizes.  Within this framework, hyperuniformity is characterized by the
equation $s=0$, defining the null hypothesis $H_0$. To test this
hypothesis, we introduce {\em maximum likelihood estimators}
$\widehat{t}_0$ of $t$ under the null hypothesis and
$(\widehat{s},\widehat{t}_1)$ of $(s,t)$ under the full model,
respectively. Under $H_0$, the estimator $\widehat{t}_0$ of $t$ is
unbiased and consistent when applied to random data, following an
exact gamma distribution for finite system sizes.  Under the full
model we observe that, asymptotically, the distribution of
$\widehat{s}$ equals a mixture of the Dirac distribution at $0$ and a
gamma distribution, while $\widehat{t}_1$ follows a gamma
distribution.  In particular, the distribution of $\widehat{s}$ has an
atom at $0$.  These findings are based on numerical solutions of the
likelihood equations and extensive simulations.

In Section \ref{seclratio}, we introduce our likelihood ratio
test. Within the asymptotic and parametric setting of Section
\ref{secML}, we analyze the ratio of the maximum likelihood of the
data under the null hypothesis to the corresponding maximum likelihood
over the full parameter space. The resulting test statistic is twice
the negative logarithm of this ratio.  Analytically, the asymptotic
null distribution of this test statistic appears to be out of
reach. However, through extensive simulations, we have determined this
distribution with high accuracy. Our empirical results indicate that
this limit null distribution consists of a mixture of an atomic
component at $0$ (inherited from $\widehat{s}$) and a gamma
distribution. Notably, this distribution already applies for moderate
system sizes. Furthermore, we rigorously demonstrate that this limit
distribution does not depend on the specific value of the parameter
$t$. This result enables a significance test for hyperuniformity,
which can be conveniently applied to spatial point patterns
encountered in real-world applications.

In Section \ref{sectestexample}, we apply the novel test to independent
thinnings of the matching process from \cite{KLY20}; see also Example
(v) in Section \ref{asymptsctt}. Introducing a thinning parameter
provides a convenient way to tune the parameter $s$. The hyperuniform
case $s=0$ corresponds to the limit case of no thinning.  Our findings
confirm that the asymptotic distribution of the test statistic applies
extremely well to this point process. Remarkably, the test performs
reliably even for a single sample with a moderately large system size.

We apply our test with a nominal significance level of 0.05 for
different values of the parameter $s$. When $s=0$, the test accurately
maintains the nominal level.  Even for $s=10^{-3}$ and a system size of
$L=150$ in two dimensions (i.e., for about $150^2$ points), the
test rejects hyperuniformity in almost all cases. Furthermore for
$s=10^{-4}$, which is an order of magnitude smaller than typical
non-hyperuniform models (see~\cite{Torquato2018, klatt_universal_2019}),
the rejection rate increases dramatically as the systems size grows; see
Table~\ref{tab:power} for more details.

We conclude the paper with a summary of the test procedure and
remarks on the fundamental principles underlying the novel test. We then
outline several interesting open problems for future research
directions. For the sake of readability, some proofs have been moved to
an appendix.

\section{Preliminaries}\label{s2}

We begin by introducing some point process terminology,
primarily drawn from Chapters 8 and 9 of \cite{LastPenrose17}.
Our setting is $\R^d$ for $d\ge 1$, equipped with
the Borel $\sigma$-field $\cB^d$ and Lebesgue measure $\lambda_d$.
Let $\cB^d_b$ denote the collection of bounded elements of $\cB^d$, and let
$\bN$ represent the system of locally finite subsets of $\R^d$.
For $\mu\in\bN$ and $B\subset\R^d$ we denote by $\mu(B)$ the
cardinality of $\mu\cap B$. By definition, $\mu(B)\in\N_0$ for each
$B\in \cB^d_b$. We define $\mathcal{N}$ as the smallest $\sigma$-field on $\bN$
that makes the mappings $\mu\mapsto\mu(B)$ measurable for each $B\in\cB^d$.

A (simple) {\em point process} is a random element $\eta$ of
$(\bN,\mathcal{N})$, defined on some fixed probability space
$(\Omega,\mathcal{F},\BP)$ with associated expectation operator $\BE$.
The process is said to be {\em stationary} if the distribution of
$\eta+x:=\{y+x:y\in\eta\}$ does not depend on $x\in\R^d$.
In this case, $\gamma:=\BE \eta([0,1]^d)$ is called the
{\em intensity} (or, in physics terminology, the {\em number density}) of $\eta$.
By Campbell's formula,
\begin{align}\label{eCampbell}
\BE \sum_{x\in\eta}f(x)=\gamma \int f(x)\,\D x
\end{align}
for each measurable function $f\colon\R^d\to[0,\infty)$. Here and in what follows,
each unspecified integral is over
the full domain.

We consider a stationary point process
$\eta$ with positive and finite intensity $\gamma$.
Additionally, we assume that $\eta$ is {\em locally square integrable}, meaning that
$\BE\eta(B)^2<\infty$ for each $B\in\cB^d_b$.
The {\em second factorial moment measure} $\alpha_2$ of $\eta$
is defined by 
\begin{align*}
\alpha_2(\cdot):=\BE \sum_{x,y\in\eta}\I\big{\{}x\in[0,1]^d,y-x\in\cdot\big{\}},
\end{align*}
where $\I\{ \cdot \}$ stands for the indicator function.  
It is often more convenient to work
with the {\em reduced second factorial moment measure} $\alpha^!_2$ of $\eta$, defined
by
\begin{align*}
\alpha^!_2(\cdot):=\BE \sideset{}{^{\ne}}\sum_{x,y\in\eta}\I\big{\{}x\in[0,1]^d,y-x\in\cdot\big{\}}.
\end{align*}
Here, the superscript $\ne$ indicates that the summation is taken over
distinct pairs.  From the definition and \eqref{eCampbell}, it follows
that $\alpha_2=\alpha^!_2+\gamma \delta_0$, where
$\delta_0$ is the {\em Dirac measure} at $0\in\R^d$. 
Note that $\alpha^!_2$ is a locally finite measure on $\R^d$,
which is invariant under reflections at the origin. Moreover, it satisfies the equation
\begin{align}\label{edisint}
\BE \sideset{}{^{\ne}}\sum_{x,y\in\eta}\I\{(x,y)\in\cdot\}
= \iint \I\{(x,x+y)\in\cdot\}\,\alpha^!_2(\D y)\,\D x.
\end{align}
In many cases, the measure $\alpha^!_2$ admits a density with respect to the Lebesgue measure, 
which gives
\begin{align*}
\alpha^!_2(\D y)=\gamma^2g_2(y)\,\D y,
\end{align*}
where $g_2\colon \R^d\to[0,\infty)$ is
the so-called {\em pair correlation function} of $\eta$. This function
is measurable and locally integrable,  and it can be assumed
to satisfy the symmetry property $g_2(x)=g_2(-x)$ for each $x\in\R^d$.
For any $B\in\cB^d$, it follows from \eqref{edisint}  that
\begin{align*}
\BE \eta(B)^2=\gamma\lambda_d(B)+\iint\I\{x\in B,x+y\in B\}\,\D x\,\alpha^!_2(\D y);
\end{align*}
see \cite[(4.25) and (8.8)]{LastPenrose17}.
If $B$ is bounded, then the variance of $\eta(B)$ is given by
\begin{align*}
\BV\eta(B)=\gamma\lambda_d(B)+\int \lambda_d(B\cap (B+y))\,\alpha^!_2(\D y)-\gamma^2\lambda_d(B)^2.
\end{align*}
Since $\int\lambda_d(B\cap (B+y))\,\D y=\lambda_d(B)^2$, an alternative representation of the variance is
\begin{align}\label{e1.11}
\BV\eta(B)=\gamma\lambda_d(B)+\int \lambda_d(B\cap (B+y))\,\beta^!_2(\D y),
\end{align}
where the {\em signed measure} $\beta^!_2$ is defined as
\begin{align}\label{ebeta2!}
\beta^!_2:=\alpha^!_2-\gamma^2\lambda_d.
\end{align}
The measure $\beta_2:=\beta^!_2+\gamma\delta_0=\alpha_2-\gamma^2\lambda_d$
is the {\em covariance measure} of $\eta$; see \cite[Chapter 9]{Bremaud20}.  
In fact, $\beta^!_2$ is the second reduced
factorial cumulant measure of $\eta$,
which is a special case of a locally finite measure $\beta^!_m$, defined in
\eqref{ereducedcum} for each integer $m \ge 2$.
The covariance measure is only defined on the system of bounded Borel sets of $\R^d$.
The restriction $\beta^!_2(B\cap \cdot)$
of $\beta^!_2$ onto a bounded Borel subset $B$ of $\R^d$
is a finite signed measure, which hence can be written as the
difference of two orthogonal finite measures. This allows us to
define the {\em total variation measure}
$|\beta^!_2|$,
first on $B$, and then, by consistency, on
each Borel set. The result is a locally finite measure on $\R^d$.
If the point process $\eta$ has a pair correlation function $g_2$, then
\begin{align*}
\beta^!_2(\D x)=\gamma^2(g_2(x)-1)\,\D x
\end{align*}
and
\begin{align}\label{ebeta227}
|\beta^!_2|(\D x)=\gamma^2|g_2(x)-1|\,\D x.
\end{align}
In Sections \ref{scovariance} and \ref{sclt}, we assume that
\begin{align}\label{efinite}
|\beta^!_2|(\R^d)<\infty.
\end{align}
Clearly, this condition is equivalent to $|\beta_2|(\R^d)<\infty$.

Let $W\subset\R^d$ be a convex compact set with a non-empty interior.
For each $y\in\R^d$, we have
$\lim_{r\to\infty}\lambda_d(rW)^{-1}\lambda_d(rW\cap (rW+y))=1$.
If condition \eqref{efinite} holds, it follows from \eqref{e1.11} and the dominated convergence theorem that
\begin{align}\label{easvariance}
\lim_{r\to\infty}\frac{\BV\eta(rW)}{\lambda_d(rW)}
=\gamma+\beta^!_2(\R^d).
\end{align}
The point process $\eta$ is said to be {\em hyperuniform} with respect to  $W$
if
\begin{align*}
\lim_{r\to\infty}\frac{\BV\eta(rW)}{\lambda_d(rW)}=0.
\end{align*}
If this convergence holds with $W$ chosen as the unit ball
$B_1:=\{x\in\R^d:\|x\|\le 1\}$, where $\|\cdot \|$ denotes the Euclidean norm, we simply say that
$\eta$ is hyperuniform~\cite{TorquatoStillinger2003, Torquato2018}.
If condition \eqref{efinite} holds, then
\eqref{easvariance} shows that hyperuniformity is equivalent to
\begin{align}\label{ehyperg2}
\beta^!_2(\R^d)=-\gamma.
\end{align}

For the remainder of this section, we assume that condition \eqref{efinite} holds.
Let $\ti$ denote the imaginary unit,  and let $\langle \cdot, \cdot \rangle$ represent the inner product
on $\R^d \times \R^d$. The {\em structure factor} $S$ of $\eta$ is the function on $\R^d$ defined by
\begin{align}\label{SFgeneral}
S(k):=1+ \frac{1}{\gamma} \int \te^{-\ti\langle k,x\rangle}\,\beta_2^!(\D x),\quad k\in\R^d.
\end{align}
This definition is justified by assumption \eqref{efinite}.
If $\eta$ has a pair correlation function, then the structure factor is given by
\begin{align*}
S(k)=1+\gamma \int \te^{-\ti\langle k,x\rangle}(g_2(x)-1)\,\D x,\quad k\in\R^d.
\end{align*}
Thus, $\gamma^{-1}(S-1)$ is the Fourier transform of $g_2-1$.
In the general case, $\gamma^{-1}(S-1)$ is
the Fourier transform of the covariance measure.
Since $\beta_2^!$ is reflection-invariant, 
the function $S$ is real-valued, and we have $S(k)=S(-k)$
for each $k\in\R^d$. Finally, hyperuniformity is equivalent to
\begin{align*}
S(0)=0.
\end{align*}
Throughout this paper, the symbol $0$ stands for the origin in $\R^d$
for any $d \ge 1$, which includes, in particular, the real number zero.
The specific meaning of $0$ will always be clear from the context. 

In the important special case of a {\em perturbed lattice},
assumption \eqref{efinite} does not hold; see, e.g., 
\cite{klatt_cloaking_2020}. Nevertheless, it is still possible to define the
structure factor in such cases, at least in a neighborhood
of the origin. To achieve this, we introduce the {\em Fourier transform}
$\hat{\beta}_2$ of $\beta_2$, which is a locally finite measure 
on $\R^d$ satisfying
\begin{align}\label{e:spectral}
\int (f\star g)(x) \,\beta_2(\D x)=\frac{1}{(2\pi)^d}\int \hat{f}(k)\hat{g}(-k)\,\hat{\beta}_2(\D k),
\end{align}
for all bounded measurable functions $f,g\colon\R^d\to\R$ with bounded support,
where the {\em tilted convolution}  $(f\star g)$ is defined as  
$(f\star g)(x):=\int f(x)g(x-y)\,dy$,  $y\in\R^d$.
This measure is also known as the {\em Bartlett spectral measure}
of $\eta$; see  \cite[Chapter 9]{Bremaud20}.
If $\beta_2$ has finite total variation, then
\begin{align}\label{e:7134}
\hat\beta_2(\D k)=\gamma  S(k)\,\D k,
\end{align}
where $S(k):=\gamma^{-1}\int \te^{-\ti \langle k,x\rangle}\,\beta_2(\D x)$, $k\in\R^d$.
We continue to refer to the function $S$ as the structure factor, 
even if equation \eqref{e:7134} holds only in a neighborhood of the origin.
If $\eta$  is a (stationary) perturbed lattice based on a
{\em lattice} $\mathbb{L}\subset\R^d$
and the distribution $\BQ$ of independent translations, then
\begin{align*}
\hat{\beta}_2(\D k)=(1-|\hat{\BQ}|^2(k))\,\D k
+\sum_{l\in \mathbb{L}^*\setminus \{0\}}|\hat{\BQ}(k)|^2\delta_l(\D k),
\end{align*}
where $\hat{\BQ}$ is the Fourier transform of $\BQ$, 
and $\mathbb{L}^*$ denotes the {\em dual lattice} of $\mathbb{L}$.
Hence, we may define $S(k):=1-|\hat{\BQ}(k)|^2$, except at points
in $\mathbb{L}^*\setminus \{0\}$. If $\int \|x\|^2\,\BQ(\D k)<\infty$ and
$\BQ$ is isotropic (i.e., invariant under rotations), then 
\eqref{skparabola} holds with $\alpha=2$.
We refer to \cite{gabrielli_point_2004, klatt_cloaking_2020,
Torquato2018, BjoerklundHartnick24} for more details on perturbed
lattices and their Fourier transforms.

Assuming that \eqref{skparabola} holds,
we test the hypothesis $S(0)=0$ (i.e.,\ hyperuniformity) using the
empirical {\em scattering intensity}
\begin{align}\label{scattering}
\mathcal{S}_r(k):= \frac{1}{\eta(W_r)} \bigg|\sum_{x\in\eta\cap W_r}\te^{-\ti\langle k,x\rangle}
-\I\{k=0\}\gamma r^d\bigg|^2,\quad k\in\R^d,
\end{align}
where $1/0:=0$, and  $W_r$ is a cube centered at $0$ with side length $r>0$.

\section{Asymptotic covariance structure}\label{scovariance}

In this section, we assume that
\begin{align}\label{efinite3}
\int (1+\|x\|^p)\,|\beta^!_2|(\D x)<\infty,
\end{align}
for some $p>0$, which strengthens condition \eqref{efinite}.
If, for instance, $|g_2(x)-1|\sim \|x\|^{-d-\alpha}$ as $\|x\|\to\infty$, for some $\alpha\in(0,2)$, then
\eqref{skparabola} holds and we can choose $p<\alpha$.
Since $\beta^!_2$ is locally finite, we can assume $p\le 1$ without loss of generality.

We fix a bounded and measurable function $h\colon\R^d\to[0,\infty)$ satisfying
\begin{align}\label{e974}
M_1(h)<\infty,
\end{align}
where
\begin{align*}
M_j(h):=\int h(x)^j\,\D x,\quad j=1,2.
\end{align*}
Since $h$ is bounded, it follows from \eqref{e974} that $M_2(h)<\infty$.
We also assume the existence of constants $a,b>0$ such that
\begin{align}\label{e239}
\I\{\|y\|\le  b\}\int |h(x+y)-h(x)|\,\D x\le a \|y\|^p, 
\end{align}
where $p>0$ is as in assumption \eqref{efinite3}.
Note that if \eqref{e239} holds for $p=1$ (or $p>1$) then it holds for
each $p\le 1$. This only requires replacing the constant $a$ by $ab^{1-p}$.
Note also that
\begin{align*}
\int |h(x+y)-h(x)|\,\D x\le 2M_1(h),\quad y\in\R^d.
\end{align*}
The class of measurable bounded functions $h\colon\mathbb{R}^d \rightarrow [0,\infty)$ that satisfy conditions    
\eqref{e974} and \eqref{e239} is denoted by ${\mathcal H}_p$.

Throughout this article we  take $h\in {\mathcal H}_p$ for some (suitable) $p>0$.
Since $M_2(h)<\infty$ we have by the Riemann--Lebesgue lemma
\begin{align}\label{e977}
\lim_{\|k\|\to\infty} \int h(x)^2 \te^{\ti\langle k,x\rangle}\,\D x=0.
\end{align}
For $r>0$, we define the rescaled function $h_r(x):=h(x/r)$. Observing that $\int h_r(x)^j\,\D x=M_j(h)r^d$
for $j\in\{1,2\}$, we note that
\begin{align}\label{e977a}
\frac{1}{r^d}\int h_r(x)^2 \te^{\ti\langle k,x\rangle}\,\D x=\int h(x)^2 \te^{\ti\langle rk,x\rangle}\,\D x,\quad k\in\R^d.
\end{align}
Later, we consider functions $r\mapsto k_r$ from $(0,\infty)$ to $\R^d$, abbreviated as $(k_r)$, 
satisfying $k_\infty:=\lim_{r\to\infty}k_r\ne 0$ and
\begin{align}\label{e234}
\lim_{r\to\infty} &\frac{1}{r^{d/2}}\int h_r(x) \te^{\ti\langle k_r,x\rangle}\,\D x
=\lim_{r\to\infty} \frac{1}{r^{d/2}}\int h_r(x) \te^{-\ti\langle k_r,x\rangle}\,\D x=0.
\end{align}
The space of all such functions is denoted by $\mathbf{F}$. 
We also frequently consider the extended space $\mathbf{F}\cup \{\mathbf{0}\}$,
where $\mathbf{0}$ is the zero function.

\begin{example}\label{exball}\rm Let $B_r\subset \R^d$ denote the
closed ball with radius $r$, centered at the origin, and let $\kappa_d=\pi^{d/2} / \Gamma(1+d/2)$
denote the volume of $B_1$. Let $r_d>0$ satisfy $\kappa_dr_d^d=1$, so that $B_{r_d}$ has volume 1.
Let $h$ be the indicator function of $B_{r_d}$.  Then
  $M_1(h)=M_2(h)=1$, and the rescaled function $h_r$ is the indicator function of a ball   with radius $rr_d$. 
To verify condition \eqref{e234}, we use a classical result
  concerning the Fourier transform of the indicator function of a ball;
  see, e.g., \cite[Section B.5]{Gra09}. Specifically, for each $r>0$,
\begin{align*}
\int_{B_r} \te^{\ti\langle k,x\rangle}\,\D x
=\Big(\frac{2\pi r}{\|k\|} \Big)^{d/2} J_{d/2}(\|k\|r), \quad   k\in\R^d\setminus\{0\},
\end{align*}
where $J_a\colon[0,\infty)\to\R$ is the {\em Bessel function of the first kind},
defined for $a>-1/2$ as
\begin{align*}
J_a(t):=\sum^\infty_{m=0}\frac{(-1)^m}{m!\Gamma(m+a+1)}\Big(\frac{t}{2}\Big)^{2m+a}, \quad t\ge 0.
\end{align*}
Since $\lim_{t\to\infty} J_a(t)=0$ (see ~\cite[Section B.8]{Gra09}), 
condition \eqref{e234} holds for any function $(k_r)$ with a non-trivial limit.

To verify \eqref{e239}, define $a_+:=\max(0,a)$ for $a\in\R$ and fix $y\in\R^d$.
If $x\in B_{r_d}$, but $x+y\notin B_{r_d}$, and if $\|y\|\le r_d$, the triangle inequality  implies $r_d-\|y\| \le \|x\|$. Thus,
\begin{align*}
\int&\I\{x\in B_{r_d},x+y\notin B_{r_d}\}\,\D x\\
&\le\I\{\|y\|\le r_d\}\int\I\{x\in B_{r_d}\setminus B_{(r_d-\|y\|)_+}\}\,\D x
+\I\{\|y\|>r_d\}.
\end{align*}
(Note that this inequality holds trivially if $\|y\| >r_d$).)
The first term can be bounded by
\begin{align*}
\I\{\|y\|\le r_d\}\big(\kappa_dr_d^d-\kappa_d(r_d-\|y\|)^d\big)
&\le\I\{\|y\|\le r_d\}\sum_{j=0}^{d-1}c_jr_d^j \|y\|^{d-j}\\
&\le \I\{\|y\|\le r_d\}\|y\| \sum_{j=0}^{d-1}c_jr_d^{d-1},
\end{align*}
for some positive constants $c_0,\ldots,c_{d-1}$. Since the same bound applies to
\begin{align*}
\int\I\{x\notin B_{r_d},x+y\in B_{r_d}\}\,\D x
=\int\I\{x-y\notin B_{r_d},x\in B_{r_d}\}\,\D x,
\end{align*}
it follows that \eqref{e239} holds for $p=1$, showing that $h \in {\mathcal H}_p$
for $p\le 1$. 
\end{example}

\begin{example}\label{excube}\rm Assume that $h$ is the
indicator function of the unit cube $W_1:=[-1/2,1/2]^d$, centered at the origin.
Then $M_1(h)=M_2(h)=1$, and
$h_r$ is the indicator function of the cube $W_r:=[-r/2,r/2]^d$.
A straightforward calculation yields
\begin{align*}
\int_{W_r} \te^{\ti \langle k,x\rangle}\,\D x=\prod^d_{j=1}\frac{2}{k_j}\sin\frac{k_jr}{2},\quad k=(k_1,\ldots,k_d)\in\R^d,
\end{align*}
where we define $\sin st/s:=t$ for $s=0$.
Thus,  
\eqref{e234} is valid for each function $(k_r)$ that has a limit in the set
\begin{align}\label{eA*}
A_\Box:=\bigg\{k\in\R^d:\sum^d_{i=1}\I\{|k_i|>0\}> \frac{d}{2} \bigg\}.
\end{align}
If $d=2$, then $A_\Box=\{(k_1,k_2):k_1k_2\ne 0\}$.
Furthermore, equation \eqref{e234} also holds if $k_rr\in 2\pi\Z^d$ 
for each $r>0$, in which case the integral in \eqref{e234}  vanishes.
Relation \eqref{e239} can be checked as in Example \ref{exball}.
\end{example}

\begin{example}\label{exnormal}\rm Now, assume that $h(x)=(2\pi)^{-d/2}\te^{-\|x\|^2/2}$
is the density of the standard normal distribution. Then 
\begin{align}
\int h(x) \te^{\ti\langle k,x\rangle}\,\D x&=\te^{-\|k\|^2/2},\quad k\in\R^d,\\
\int h(x)^2 \te^{\ti\langle k,x\rangle}\,\D x&=\pi^{-d/2}\te^{-\|k\|^2},\quad k\in\R^d.
\end{align}
Hence \eqref{e234} is valid for any $(k_r)$ with a non-trivial limit.
We now assert that $h\in {\mathcal H}_1$, meaning that \eqref{e239} holds for $p=1$.
To verify this, we apply the mean value theorem. For 
$x,y\in\R^d$, we obtain
$h(x+y)-h(x)=\langle  h'(x+\vartheta),y\rangle$,
where $h'$ is the gradient of $h$ and
$\vartheta\in[0,1]$ depends on $x$ and $y$. Consequently,
\begin{align*}
|h(x+y)-h(x)|=c\te^{-\|x+\vartheta y\|^2/2}|\langle x+\vartheta y,y\rangle|,
\end{align*}
where $c:=(2\pi)^{-d/2}$. It follows that
\begin{align*}
\int |h(x+y)-h(x)|\,\D x\le c\|y\|\int \|x\| \te^{-\|x+\vartheta y\|^2/2}\,\D x
+c\|y\|^2 \int \te^{-\|x+\vartheta y\|^2/2} \,\D x.
\end{align*}
Now, assume that $\|y\|\le 1$. It is easy to verify that $\|x+\vartheta y\|\ge \|x\|/2$ whenever $\|x\|\ge 2$.
By splitting the domain of integration into $\|x\|< 2$ and $\|x\|\ge 2$,
it follows that \eqref{e239} holds for $p=1$ and $b=1$.
Further details are left to the reader.
\end{example}

Similarly to \cite{Brillinger72}, we define, for each $r>0$,
a random function $T_r$ on $\R^d$ by
\begin{align}\label{eTr}
T_r(k):= \frac{1}{r^{d/2}}\sum_{x\in\eta}h_r(x)\te^{-\ti\langle k,x\rangle}
-\I\{k=0\}M_1(h)\gamma r^{d/2},\quad k\in\R^d.
\end{align}
In the language of point process theory, $T_r(k)$ is a linear function
of $\eta$.  Note that $T_r(-k)$ is the complex conjugate
$\overline{T_r(k)}$ of $T_r(k)$.  In the case of Example~\ref{excube},
$T_r$ is closely related to the scattering intensity
given in \eqref{scattering}.  In this context, the function $h$ appearing in
\eqref{eTr} is referred to as a {\em taper function}; see
\cite{Brillinger72,rajala_what_2023}. In physics, the variables defined in \eqref{eTr}
are known as {\em collective coordinates};
see \cite{uche_collective_2006}.
If $(k_r)\in\mathbf{F}\cup \{\mathbf{0}\}$, then
applying  Campbell's formula \eqref{eCampbell} along with
\eqref{e234} yields
\begin{align}\label{e590}
\lim_{r\to\infty}\BE T_r(k_r)= 0.
\end{align}
In fact, if $(k_r)=\mathbf{0}$ then $\BE T_r(k_r)=0$ for each $r>0$.
The following result describes the asymptotic behavior of the covariances of the functionals $T_r$.
The proof is given in the Appendix; see Subsection \ref{appendix1}.

\begin{proposition}\label{pcov} Suppose that $(k_r),(\ell_r)\in \mathbf{F}\cup \{\mathbf{0}\}$,
and set $k:=k_\infty$ and $\ell:=\ell_\infty$. If $k\ne \ell$, then
\begin{align}\label{eascov1}
\lim_{r\to\infty}\BE T_r(k_r)\overline{T_r(\ell_r)}=0.
\end{align}
Furthermore,
\begin{align}\label{eascov1a}
\lim_{r\to\infty}\BE |T_r(k_r)|^2=\gamma M_2(h) S(k).
\end{align}
\end{proposition}

Later, we will use the following non-asymptotic result.

\begin{corollary}\label{c153} Assume that \eqref{efinite} holds. Let $r>0$ and $k\in\R^d\setminus\{0\}$. Then
\begin{align}\label{e107}\notag
\BE |T_r(k)|^2&=\gamma M_2(h)S(k)
+\frac{1}{r^d} \iint h_r(x)(h_r(x+y)-h_r(x))
\te^{\ti\langle k,y\rangle}\,{\rm d} x\,\beta^!_2(\text{{\rm d}} y)\\
&\quad+\frac{\gamma^2}{r^d}\bigg|\int h_r(x)\te^{-\ti\langle k,x\rangle}\, \text{{\rm d}} x\bigg|^2.
\end{align}
\end{corollary}
\begin{proof}
By \eqref{e389}, we have
\begin{align*}
\BE |T_r(k)|^2=\gamma M_2(h)
+\frac{1}{r^d} \iint h_r(x)h_r(x+y)\te^{\ti\langle k,y\rangle}\,\D x\,\alpha^!_2(\D y).
\end{align*}
Using \eqref{ebeta2!}, this can be rewritten as
\begin{align*}
\BE |T_r(k)|^2=\gamma M_2(h)+A_r+B_r,
\end{align*}
where
\begin{align*}
A_r&:=\frac{1}{r^d} \iint h_r(x)h_r(x+y)
\te^{\ti\langle k,y\rangle}\,\D x\,\beta^!_2(\D y),\\
B_r&:=\frac{\gamma^2}{r^d}\iint h_r(x)h_r(x+y)e^{-\ti\langle k,y\rangle}\,\D x\,\D y.
\end{align*}
Since $\te^{\ti\langle k,y\rangle}=\te^{\ti\langle k,x\rangle}\te^{-\ti\langle k,x+y\rangle}$, it follows that 
$B_r$ corresponds to the third term on the right-hand side of \eqref{e107}.
Moreover, we obtain
\begin{align*}
A_r=M_2(h)\int\te^{\ti\langle k,y\rangle}\,\beta^!_2(\D y)
+\frac{1}{r^d} \iint h_r(x)(h_r(x+y)-h_r(x))
\te^{\ti\langle k,y\rangle}\,\D x\,\beta^!_2(\D y),
\end{align*}
which concludes the proof.
\end{proof}

\begin{remark}\label{p:098}\rm In general, Proposition \ref{pcov} does not hold
for functions  $(k_r),(\ell_r)$ that tend to zero in a non-trivial manner.
For example, consider $(k_r)=\mathbf{0}$ and $\ell_r=\pi/r$.
Then the first term on the right-hand side of \eqref{2078}
equals $\gamma\int h(x)^2\te^{\ti\langle \pi,x\rangle}\,\D x$.
Combining this with \eqref{2079}, we obtain
\begin{align*}
\lim_{r\to\infty}\BE T_r(0)\overline{T_r(\ell_r)}=(\gamma+\beta^!_2(\R^d))\int h(x)^2\te^{\ti\langle \pi,x\rangle}\,\D x,
\end{align*}
which is not equal to $\gamma M_2(h)$. 
\end{remark}

For $k\in\R^d$ and $r>0$, we denote
by $T_{r,1}(k)$ the real part and by $T_{r,2}(k)$
the imaginary part of $T_r(k)$. 
Given $(k_r),(l_r)\in\mathbf{F}\cup\{\mathbf{0}\}$, we define
\begin{align}\label{e:ascov}
  \sigma_{m,n}(k,\ell):=\lim_{r\to\infty}\BE T_{r,m}(k_r)T_{r,n}(\ell_r),
  \quad m,n\in\{1,2\},
\end{align}
where $k:=k_\infty$ and $\ell:=\ell_\infty$. This definition is justified by
Proposition \ref{pcov}. 
We also define the set
\begin{align}
A_h:=\{k_\infty:(k_r)\in\mathbf{F}\}\cup\{\mathbf{0}\}, 
\end{align}
noting that $A_h=-A_h$.

\begin{corollary}\label{c945} Let $k,\ell\in A_h$ such that either $k=\ell$ or $k\notin\{\ell,-\ell\}$.  Then
\begin{align}\label{e3.45}
\sigma_{m,n}(k,\ell)=\I\{m=n,k=\ell\}\frac{\gamma M_2(h)S(k)}{2},\quad m,n\in\{1,2\}.
\end{align}
\end{corollary}
\begin{proof}
By \eqref{eascov1a},
\begin{equation*}
\sigma_{1,1}(k,k)+\sigma_{2,2}(k,k)=\gamma M_2(h)S(k).
\end{equation*}
We also have the trivial identity $\sigma_{1,2}(k,k)=\sigma_{2,1}(k,k)$.

Applying \eqref{eascov1} with $\ell=-k$ and using the property $T_r(-k)=\overline{T_r(k)}$, we obtain
\begin{align*}
\sigma_{1,1}(k,k)-\sigma_{2,2}(k,k)=\sigma_{1,2}(k,k)+\sigma_{2,1}(k,k)=0.
\end{align*}
Thus, \eqref{e3.45} holds for $k=\ell$.
Now, assume that  $k\notin\{\ell,-\ell\}$. Applying
\eqref{eascov1} with $\ell$ and $-\ell$, it follows that
$\sigma_{m,n}(k,\ell)=0$ for all $m,n\in\{1,2\}$.
\end{proof}

\begin{remark}\label{remadmissable}\rm
Assume that $h$ is the cubic taper function from Example \ref{excube}.
Let $r>0$ and $k=(k_1,\ldots,k_d)\in \R^d\setminus\{0\}$.
Assuming \eqref{efinite3}, it follows
from Corollary \ref{c153}
and a straightforward computation that
\begin{align}\label{e5001}
\BE |T_r(k)|^2=\gamma S(k)+f_r(k)
+ \frac{\gamma^2}{r^d} \bigg|\prod^d_{j=1}\int^{r/2}_{-r/2}\te^{-\ti k_js}\,\D s\bigg|^2,
\end{align}
where
\begin{align*}
f_r(k):=\frac{1}{r^d} \iint h_r(x)(h_r(x+y)-h_r(x))
\te^{\ti\langle k,y\rangle}\,\D x\,\beta^!_2(\D y).
\end{align*}
From the proof of Proposition~\ref{pcov}, it follows that 
\begin{align*}
\sup_{k\in \R^d\setminus\{0\}}|f_r(k)| \to 0,\quad \text{as $r\to\infty$}.
\end{align*}
If $k_j=2\pi n/r$ for some $j\in[d]$ and $n\in\N$, then the
last term on the right-hand side of \eqref{e5001} vanishes.
Let 
\begin{align}\label{Kr'} 
K'_r:=\{k \in \R^d: k_{j}=2\pi n/r \text{ for some } j \in [d] \text{  and } n \ge 1\}. 
\end{align}
Then
\begin{align*}
\sup_{k\in K'_r}\big|\gamma^{-1}\BE |T_r(k)|^2- S(k)\big|
\to 0,\quad \text{as $r\to\infty$}.
\end{align*}
On the other hand,  we may choose $k_j\equiv k_j(r)=\pi/r$ for each $j\in[d]$;
see also Remark~\ref{p:098}. Then, for any $\varepsilon >0$,
equation \eqref{e5001} implies that
\begin{align*}
\sup_{k\in \R^d\setminus\{0\},\|k\|\le \varepsilon}\big|\gamma^{-1}\BE |T_r(k)|^2- S(k)\big|
\to \infty,\quad \text{as $r\to\infty$}.
\end{align*}

In our statistical setting we will work with a cubic taper function.
For fixed $r$, we will always choose the wave vectors 
from the set $K_r\subset K'_r$, which is defined by
\eqref{eq:reciprocal} after replacing the system‑size label $L$ by $r$
(i.e., $K_r=K_L$).
This choice is commonly adopted in the physics
literature~\cite{Torquato2018,hansen_theory_2013}, see also
\cite{HGBL22} for a recent mathematical discussion and \cite{liu_impact_2025} for a related study in physics, 
and it is also supported by Corollary \ref{remgood}.
\end{remark}

\section{A central limit theorem}\label{sclt}

In this section, let $\eta$ be a stationary point process with intensity $\gamma$ satisfying
\begin{align}\label{emoment}
\BE \eta(B)^m<\infty
\end{align}
for each bounded Borel set $B \subset \R^d$ and each $m \in \N$.
We assume $h\in \mathcal H_p$ for some $p > 0$, which will be fixed in Theorem~\ref{clt}.
Let $Z_k$, $k\in\R^d$, be independent  centered $\C$-valued normally
distributed random variables. The components of $Z_k$
are assumed to be independent and to have variance $\gamma M_2(h)S(k)/2$.
Here and in the sequel, a normally distributed random variable with variance $0$ is
almost surely constant. 
For a given  $n\in\N$, we consider $(k_{1,r})_{r>0},\ldots,(k_{n,r})_{r>0}\in\mathbf{F}\cup\{\mathbf{0}\}$
and set $k_i:=\lim_{r\to\infty}k_{i,r}$. In view of Corollary \ref{c945}, we assume that
\begin{align}\label{e:0518}
\{k_i,-k_i\}\cap \{k_j,-k_j\}=\emptyset,\quad i\ne j.
\end{align}
Our primary interest is in the multivariate central limit theorem
\begin{align}\label{eclt}
(T_r(k_{1,r}),\ldots,T_r(k_{n,r}))\overset{d}{\longrightarrow}(Z_{k_1},\ldots,Z_{k_n}), \quad \text{as $r\to\infty$},
\end{align}
where $\overset{d}{\longrightarrow}$ denotes convergence in distribution.
In the case $d=1$, Brillinger (see ~\cite[Theorem 4.2]{Brillinger72}) proved that this convergence
holds under a suitable assumption on the so-called (reduced) {\em cumulant measures} of
$\eta$. We shall extend Theorem 4.2 of \cite{Brillinger72} to general dimensions,
using similar methods.

To state our result, we introduce some concepts from the theory of point processes (see, e.g., \cite{DaleyVereJones}).
For each integer  $m$, the $m$-th {\em factorial moment measure} of $\eta$
is the  measure on $(\R^d)^m$, defined by
\begin{align*}
\alpha_m(\cdot):=\BE \sideset{}{^{\ne}}\sum_{x_1,\ldots,x_m\in\eta}
\I\{(x_1,\dots,x_m)\in \cdot\}.
\end{align*}
Here, $\ne$ indicates summation over $m$-tuples with
pairwise distinct entries. By assumption \eqref{emoment}, $\alpha_m(B)<\infty$ for each
bounded Borel set $B\subset (\R^d)^m$. 
If $\alpha_m$ has a density $\rho_m$ with respect to Lebesgue measure on $(\R^d)^m$, then $\rho_m$ is called 
the $m$-th {\em correlation function} of $\eta$.
For $j\in[m]:=\{1,\dots,m\}$, we denote by $\Pi_{m,j}$
the system of all partitions $\pi$ of $[m]$ whose cardinality
$|\pi|$ equals $j$. The
$m$-th {\em factorial cumulant measure} $\beta_m$ of $\eta$
is the signed measure on $(\R^d)^m$ given by
\begin{align}\label{deffaccum}
\beta_m(B_1\times\cdots\times B_m):=\sum^m_{j=1}(-1)^{j-1}(j-1)!\sum_{\pi\in\Pi_{m,j}}
\prod_{I\in\pi} \alpha_{|I|}(\times_{i\in I}B_i)
\end{align}
for bounded Borel sets $B_1,\ldots,B_m\subset\R^d$. Note that
$\beta_m$ is only defined on bounded Borel subsets of $(\R^d)^m$.
An equivalent definition is
\begin{align}\label{ecummeasure}
\beta_m(B_1\times\cdots\times B_m)
=\frac{\partial^m}{\partial u_1\cdots\partial u_m}\log G\Big(1+\sum_{j=1}^mu_j\I_{B_j}
\Big)\Big|_{u_1=\cdots= u_m=0},
\end{align}
where each of the partial derivatives is to be taken from the left. Here,
the {\em generating functional} $G$ of $\eta$ is given by
\begin{align*}
G(f):=\BE \prod_{x\in\eta}f(x)
\end{align*}
for each  measurable function $f\colon\R^d\to[0,\infty)$.
For instance, we have $\beta_1(\D x)=\gamma \D x$ and
\begin{align}\label{ebeta2}
\beta_2(\D (x,y))=\alpha_2(\D(x,y))-\gamma^2\D x\D y.
\end{align}

Assume that $m\ge 2$. Since $\eta$ is stationary, $\beta_m$ is invariant under joint shifts,
and there exists a locally finite signed measure
$\beta^!_m$ on $(\R^d)^{m-1}$ (the $m$-th reduced
factorial cumulant measure of $\eta$), satisfying
\begin{align}\label{ereducedcum}
\beta_m(\cdot)=
\iint\I\{(x,x_1+x,\ldots,x_{m-1}+x)\in \cdot\}\,\beta^!_m(\D (x_1,\ldots,x_{m-1}))\,\D x.
\end{align}
If $m=2$, then \eqref{ebeta2} shows that
\eqref{ereducedcum} is consistent with \eqref{ebeta2!}.
As in the case $m=2$, we can introduce the {\em total variation measure}
$|\beta^!_m|$ of $\beta^!_m$, which is a locally finite measure on
$(\R^d)^{m-1}$.

The key assumption in this section is
\begin{align}\label{Bmixing}
|\beta^!_m|((\R^d)^{m-1})<\infty,\quad m\ge 3.
\end{align}
Together with $|\beta^!_2|(\R^d)<\infty$, this condition is known as
{\em Brillinger mixing}; see \cite{HeinrichSchmidt85}.
Later, we shall assume \eqref{Bmixing} together with \eqref{efinite3},
which slightly strengthens the Brillinger mixing property.
In one dimension, the author of \cite{Brillinger72} makes a slightly stronger
assumption (see display (2.14) in \cite{Brillinger72}).

The proof of the following central limit theorem is given in the Appendix; 
see Subsection \ref{appendix2}.

\begin{theorem}\label{clt} Let $\eta$ be a stationary point process
satisfying \eqref{emoment}, \eqref{efinite3}, and \eqref{Bmixing}, 
where $(k_{1,r})_{r>0},\ldots,(k_{n,r})_{r>0}\in\mathbf{F}\cup\{\mathbf{0}\}$
satisfy \eqref{e:0518} and $k_i:=\lim_{r\to\infty}k_{i,r}$.
Then the multivariate
central limit theorem \eqref{eclt} holds.
\end{theorem}

The following corollary is crucial for our statistical applications.
It follows from Theorem \ref{clt} and the comments made in 
Example \ref{excube}.

\begin{corollary}\label{remgood} Let $h$ be the cubic taper function from
Example \ref{excube}. Let $n\in\N$ and suppose that
$k_1,\ldots,k_n\in \R^d\setminus\{0\}$ satisfy \eqref{e:0518}.
Suppose that $N_j\colon(0,\infty)\to \Z^d$, $j\in[n]$, are such that
$2\pi N_j(r)/r\to k_j$ as $r\to\infty$. Then
\begin{align}\label{egoodpp}
\bigg(T_r\Big(\frac{2\pi}{r} N_1(r)\Big),\ldots, T_r\Big( \frac{2\pi}{r} N_n(r)\Big)\bigg)
\overset{d}{\longrightarrow}(Z_{k_1},\ldots,Z_{k_n}), \quad \text{as $r\to\infty$}.
\end{align}
\end{corollary}

\begin{remark}\label{r:BYY}\rm Assume that $\eta$ has fast decaying correlation 
functions of all orders, in the sense of \cite[(1.10)]{BYY}. It was proved in
\cite{BYY} (and in \cite{NazarovSodin12} for $d=2$) that $\eta$ is Brillinger mixing. 
Furthermore, the case $p=q=1$ of \cite[(1.10)]{BYY}, together with
our assumption $\int |g_2(x)-1|\,\D x<\infty$,
implies that
$\int \|x\|^p|g_2(x)-1|\,\D x<\infty$ for every $p>0$ (cf.\  \eqref{ebeta227} and
\eqref{efinite3}). Therefore, the assumptions of  Theorem \ref{clt} 
are satisfied, and the CLT \eqref{eclt} applies. The paper \cite{BYY}
provides numerous examples of stationary point processes with rapidly decaying correlations.
If $h$ is the cubic taper function one might try to apply 
\cite[Theorem 1.13]{BYY} directly. Since the function $f$ used there is subject to the
same scaling as the cube, this would require an exponential 
function as a score function, a choice that is formally not permitted in this theorem.
\end{remark}

We conclude the section with two general examples of stationary point processes 
satisfying assumption \eqref{Bmixing}, along with the Ginibre process as a specific example.

\begin{example}\label{expoissoncluster}\rm
Suppose that $\eta$ is a stationary {\em Poisson cluster process}
as in \cite[Exercise 8.2]{LastPenrose17}; see also
\cite[Proposition 12.1.V]{DaleyVereJones}. Let $\BQ$ denote
the cluster distribution relative to the cluster center
(a probability measure on $\bN$), and assume that
the cluster size has moments of all orders, i.e.,
\begin{align}\label{e104}
\int \mu(\R^d)^m\,\BQ(\D\mu)<\infty,\quad m\in\N.
\end{align}
Let $\gamma_0$ denote the intensity of the underlying
stationary Poisson process, and
let $\chi$ be a point process with distribution $\BQ$.
From \cite[Exercise 5.6]{LastPenrose17} or
\cite[Proposition 12.1.V]{DaleyVereJones}, it follows that 
\begin{align*}
\log G(1+f)=\gamma_0\sum^\infty_{n=1}\frac{1}{n!}
\int
\bigg[\BE\sideset{}{^{\ne}}\sum_{x_1,\ldots,x_n\in\chi}f(x_1+x)\cdots f(x_n+x)\bigg]\D x
\end{align*}
for any measurable function $f\colon\R^d\to[-1,\infty)$.
Let $\alpha^\BQ_n$, $n\in\N$, denote the factorial moment measures of
$\chi$. Then, the above can be rewritten as 
\begin{align*}
\log G(1+f)=\gamma_0\sum^\infty_{n=1}\frac{1}{n!}
\int \bigg[\int f(x_1+x)\cdots f(x_n+x)\,\alpha^\BQ_n(\D (x_1,\ldots,x_n))\bigg]\,\D x.
\end{align*}
It follows from \eqref{ecummeasure} that the factorial cumulant
measures of $\eta$ are given by
\begin{align*}
\beta_n(\cdot)=\gamma_0\int \I\{(x_1+x,\ldots, x_n+x)\in\cdot\}\,\alpha^\BQ_n(\D (x_1,\ldots,x_n))\,\D x,
\quad n\ge 2.
\end{align*}
Therefore, the reduced factorial cumulant
measures of $\eta$ take the form
\begin{align*}
\beta^!_n(\cdot)=\gamma_0\int \I\{(x_2-x_1,\ldots, x_n-x_1)\in\cdot\}\,\alpha^\BQ_n(\D (x_1,\ldots,x_n)),
\quad n\ge 2.
\end{align*}
These measures are non-negative, and \eqref{e104} ensures that  $\beta^!_n((\R^d)^{n-1})<\infty$.
For example, suppose the points $X_{n,1},\ldots,X_{n,n}$
of $\chi$, conditionally on $\chi(\R^d)=n$, are independent
and identically distributed, and that $\BE \|X_{n,2}-X_{n,1}\|\le c$
for some constant $c$ independent of $n$. Then, under assumption
\eqref{e104}, a straightforward calculation shows
that $\int \|x\|\,\beta^!_2(\D x)<\infty$. 
In fact, it would suffice in this case to assume  \eqref{e104} only for $m=2$.
\end{example}

Let $K\colon \R^d\times \R^d \to\C$ be a continuous and positive semi-definite {\em kernel}. Fix $\alpha\ge -1$, 
and assume that $\eta$ is an {\em $\alpha$-determinantal} point process with kernel $K$.
This  means that the $n$-th correlation functions  of $\eta$ are given by
\begin{align}\label{edet}
\rho_n(x_1,\ldots,x_n)=\sum_{\pi\in \Sigma_n}\alpha^{n-\#\pi}
\prod^n_{j=1}K(x_j,x_{\pi(j)}),\quad x_1,\ldots,x_n\in\R^d,
\end{align}
where $\Sigma_n$ is the set of all permutations of $[n]$, and
$\#\pi$ denotes the number of {\em cycles} in the permutation $\pi\in\Sigma_n$.
In particular, we have $\rho_1(x)=K(x,x)$ and
\begin{align}\label{e2087}
\rho_2(x,y)=K(x,x)K(y,y)+\alpha |K(x,y)|^2,
\end{align}
using the symmetry $K(y,x)=\overline{K(x,y)}$.
Note that $\alpha$ has a different meaning than in equation
\eqref{skparabola}. We hope that this usage of well-established
notation will not cause any confusion.
When $\alpha=-1$, the process $\eta$ is called {\em determinantal}.
The existence of $\eta$ is a non-trivial matter and can be guaranteed only
for specific values of $\alpha$ (including $\alpha=-1$) and under additional conditions on the kernel $K$;
see \cite{ShiraiTakahashi03,CamDec10,LMR15}. Under our assumptions, the
correlation functions \eqref{edet} determine the distribution of $\eta$.

\begin{example}\rm Suppose that $K\colon\R^d\times\R^d\to\C$ 
is continuous and positive semi-definite, and satisfies $\int |K(0,x)|\,\D x<\infty$.
Since $K$ is positive semi-definite, this implies $\int |K(0,x)|^2\,\D x<\infty$.
Assume that $\eta$ is a stationary determinantal process with kernel $K$.
Then $K(x,x)$ is constant almost everywhere with respect to Lebesgue measure,
and this constant equals the intensity $\gamma$ of $\eta$. Moreover, since
$\rho_2(x,y)=\rho_2(0,y-x)=\rho_2(0,x-y)$, equation \eqref{e2087} implies
\begin{align}\label{e2089}
|K(x,y)|=|K(0,y-x)|=|K(0,x-y)|.
\end{align}
It was shown in \cite{Heinrich16} (see also  
\cite{BiscioLavancier2016} for the case $\alpha=-1$)
that $\eta$ is Brillinger mixing.
If $K(0,0)>0$, then it follows 
from \eqref{e2087} that the pair correlation function of $\eta$
can be chosen as
\begin{align*}
g_2(x)=1+\alpha \frac{|K(0,x)|^2}{|K(0,0)|^2},\quad x\in\R^d.
\end{align*}
Thus, our assumption \eqref{efinite3} is equivalent to
$\int \|x\|^p|K(0,x)|^2\,\D x<\infty$.
Note that the hyperuniformity condition \eqref{ehyperg2} means that
\begin{align}\label{ehyperK}
\int |K(0,x)|^2\,\D x=K(0,0).
\end{align}
\end{example}

\begin{example}\label{exGinibre}\rm Assume that $d=2$,  and consider the 
 kernel 
\begin{align*}
K(x,y):=\pi^{-1}\exp\big[x\bar y-|x|^2/2-|y|^2/2\big],
\end{align*}
where  $x,y \in \C$.
Let $\eta$ be the {\em Ginibre process}, a determinantal process
with kernel $K$; see, e.g., \cite{HKPY10}.
It is straightforward to verify that  the correlation functions \eqref{edet} are invariant  under
joint shifts, so $\eta$ is stationary.
Moreover, we have $K(0,0)=\pi^{-1}$ and
\begin{align*}
|K(x,y)|^2=\pi^{-2}\te^{-|x-y|^2},\quad x,y\in\C.
\end{align*}
Therefore, equation \eqref{ehyperK} holds, and $\eta$ is hyperuniform.
In addition, assumptions \eqref{efinite3} and  \eqref{Bmixing} are satisfied.
\end{example}

\section{Asymptotic behavior of the scattering intensity}\label{asymptsctt}

As before, let $\eta$ be a stationary, locally square-integrable point process
on $\R^d$ that satisfies assumption \eqref{efinite3}. 
From now on, we additionally assume that $\eta$ is {\em ergodic}; see
\cite[Chapter 8]{LastPenrose17} for details. In this section, we focus on the scattering intensity as defined in \eqref{scattering}.
We fix the cubic taper function $h$ as in Example \ref{excube}. With this choice,
equation \eqref{scattering} can be rewritten as
\begin{align}\label{e208}
\mathcal{S}_r(k)=\frac{r^d}{\eta(W_r)}|T_r(k)|^2,\quad k\in\R^d,
\end{align}
where $T_r(k)$ is given in \eqref{eTr}. Under assumption \eqref{efinite3},
Proposition~\ref{pcov} shows that $\lim_{r\to\infty} \BE |T_r(k_r)|^2=\gamma S(k_\infty)$, for 
$(k_r)$ as in Example~\ref{excube}.

Motivated by Corollary \ref{remgood}, we introduce
the following definition to control the asymptotic
behavior of $\mathcal{S}_r(k_r)$.
Let $Z_k$, for $k\in\R^d\setminus\{0\}$, be independent,  centered, complex-valued Gaussian random variables. The real and imaginary parts of $Z_k$ 
are assumed to be independent and to have variance $\gamma S(k)/2$.
We call $\eta$ {\em good} if the following condition holds
for every integer $n$ and every choice of 
$k_1,\ldots,k_n\in \R^d\setminus\{0\}$ satisfying \eqref{e:0518}:
If $N_j\colon(0,\infty)\to \Z^d$, $j\in[n]$, are such that
$2\pi N_j(r)/r\to k_j$ as $r\to\infty$, then \eqref{egoodpp} holds.

\begin{remark}\label{remgood2}\rm
We believe that the class of good point processes is rather large.
However, rigorous results concerning the  asymptotic normality
of empirical Fourier transforms for general point processes remain scarce.
Our Theorem \ref{clt} extends a result from \cite{Brillinger72} to arbitrary dimensions.
The authors of \cite{MugRenshaw96} mention that Brillinger's approach can be extended 
to two dimensions, but they neither provide assumptions nor a rigorous proof.
Our proof benefits from the classical work 
\cite{HeinrichSchmidt85}, where a CLT for certain shot-noise processes was established.
In all of these cases, the method of cumulants is employed. 
This technique was also used in~\cite{BYY}  to prove a CLT 
for specific functionals of point processes. Although the result
in~\cite{BYY} does not directly apply in our setting, it
includes  many examples of Brillinger mixing point processes;
see Remark \ref{r:BYY}. 
In the recent preprint \cite{MastrilliBL25}, the authors establish a CLT
for linear statistics of Brillinger mixing point processes, replacing  the
trigonometric functions $x\mapsto \te^{\ti\langle k,x\rangle}$
with scaled versions of smooth, rapidly decaying functions.
The even more recent preprint \cite{Mastrilli25} derives a  CLT for perturbed lattices.
Both approaches again rely on the cumulant method.

Suitable mixing conditions may provide
an alternative to the assumptions in Theorem~ \ref{clt}.
In fact, given our results of Section~\ref{scovariance},
one might apply \cite[Theorem 1]{BiscioWaage19}.
We do not pursue this direction further here.
\end{remark}

Let $E_k$, for $k\in\R^d \setminus\{0\}$, be independent  exponentially
distributed random variables with mean $S(k)$.
(If $S(k)=0$, then $E_k=0$ almost surely.)
The following result highlights a key property of good point processes that is crucial for our test:
the scattering intensities \eqref{scattering}
converge in distribution to independent exponentially distributed
random variables.

\begin{proposition}\label{tiidexp} Assume that $\eta$ is good.
Let $n\in\N$, and let
$k_1,\ldots,k_n\in \R^d\setminus\{0\}$ satisfy \eqref{e:0518}.
Suppose $N_j\colon(0,\infty)\to \Z^d$, $j\in\{1,\ldots,n\}$, satisfy
$2\pi N_j(r)/r\to k_j$ as $r\to\infty$.
Then
\begin{align*}
\bigg(\mathcal{S}_r\Big(\frac{2\pi}{r} N_1(r)\Big),\ldots,\mathcal{S}_r\Big(\frac{2\pi}{r} N_n(r)\Big)\bigg)
\overset{d}{\longrightarrow}(E_{k_1},\ldots,E_{k_n})
\quad \text{as $r\to\infty$}.
\end{align*}
\end{proposition}
\begin{proof}
Recall equation \eqref{e208}.
By the spatial ergodic theorem for point processes
(see \cite[Corollary 10.19]{Kallenberg} or \cite[Theorem 8.14]{LastPenrose17}), we have
$r^{-d}\eta(W_r)\to\gamma$ almost surely as $r\to\infty$.
Since $\eta$ is good, the continuous
mapping theorem (see \cite[Theorem 4.27]{Kallenberg}) implies
\begin{align*}
\bigg(\mathcal{S}_r\Big(\frac{2\pi}{r} N_1(r)\Big),\ldots,\mathcal{S}_r\Big(\frac{2\pi}{r} N_n(r)\Big)\bigg)
\overset{d}{\longrightarrow}(\gamma^{-1}|Z_{k_1}|^2,\ldots,\gamma^{-1}|Z_{k_n}|^2)
\,\text{as $r\to\infty$}.
\end{align*}
The random variables $|Z_{k_1}|^2,\ldots,|Z_{k_n}|^2$ are independent.
Let $X$ and $Y$ be independent standard normal random variables. Using
$\overset{d}{=}$ to denote equality in distribution, we have
\begin{align*}
Z_k\overset{d}{=}\frac{\sqrt{\gamma S(k)}}{\sqrt{2}}(X+\ti Y).
\end{align*}
Hence,
\begin{align*}
\frac{1}{\gamma} |Z_{k}|^2\overset{d}{=}\frac{S(k)}{2}(X^2+Y^2).
\end{align*}
It is well known that $X^2+Y^2$ follows an exponential distribution with
mean $2$, which completes the proof.
\end{proof}

To support and complement the claims made in Remark~\ref{remgood2}, we
have simulated several representative examples of both non-hyperuniform
and hyperuniform point processes:

\begin{enumerate}

\item[(i)] The stationary Poisson process; see \cite{LastPenrose17}.

\item[(ii)] A Poisson cluster process, as described in
  Example~\ref{expoissoncluster}. This process is characterized by the
  following properties. Starting from a realization of a homogeneous
  Poisson point process, each point serves as the center of a cluster.
  The clusters are independent and identically   distributed, centered
  at these ``parent points'', and consist of a Poisson-distributed
  number of   ``children'' points.  Each child is drawn from a centered
  normal distribution in $\R^d$ with identity covariance matrix (wrapped
  up on the torus). The resulting point process, which includes only the
  children and not the parents, is known as the (modified) {\em Thomas
  process}. In our simulations, the mean number of points per cluster is
  set to~10.

\item[(iii)] A {\em random adsorption process} (RSA)  close to
  saturation; see~\cite{zhang_precise_2013}.  Realizations are generated
  by sequentially placing spheres (or intervals/disks in lower
  dimensions) randomly into space under  a non-overlap constraint.  We
  use the simulation procedure described in \cite{zhang_precise_2013},
  conditioning on a fixed number of points per sample, which determines
  the volume fraction of the non-overlapping spheres. The volume
  fractions used are 74.7\% for $d=1$, 54.7\% for $d=2$, and 38.4\% for
  $d=3$. 

\item[(iv)] The uniformly randomized lattice (URL),
  see~\cite{klatt_cloaking_2020}. In this model, each point of a
  stationarized lattice is independently shifted within its
  corresponding unit cell, producing a stationary point process. 

\item[(v)] Matching of the lattice $\Z^d$ with a Poisson process
of intensity $\rho>1$ \cite{KLY20}. The matched Poisson points
form a hyperuniform point process. In our simulation,
we choose $\rho=3.0$. 

\end{enumerate}

The processes described in (i)-(iii) are non-hyperuniform, whereas those
in (iv) and (v) exhibit hyperuniformity.  In the latter two models, the
lattice is stationarized, i.e., it is shifted by a random vector
uniformly distributed in the cube $[0,1)^d$. Remarkably, in both cases,
the asymptotic exponential distribution of the scattering intensity is
achieved with striking accuracy even at relatively small system sizes
(see below).

\begin{remark}\rm
The processes listed in (i), (ii) and (iv) satisfy our general
assumption \eqref{efinite3}; in fact,  the stronger assumption
\eqref{efinite2}  holds in each case. (For the Poisson cluster process,
this requires a suitable moment assumption.) The same applies to the RSA
process in (iii), although a formal proof does not yet appear to exist
in the literature. A perturbed lattice, by contrast, does not generally
satisfy  condition \eqref{efinite}. To ensure that a perturbed lattice
fulfills \eqref{efinite}, the perturbation must be
chosen in a specific way, a property shared by the URL model;
see~\cite{klatt_cloaking_2020} for details.  If the perturbation is
isotropic and has a finite second moment, then \eqref{skparabola} holds
for $\alpha=2$ in a neighborhood of the origin. For the matching
process, \eqref{efinite} does not hold, at least not in the stationary
version used here. Nevertheless, the structure factor is well-defined
for each wave vector outside the reciprocal lattice of $\Z^d$. According
to \cite[Proposition 8.1]{KLY20}, the stationarized matching process
possesses a pair correlation function $g_2$. However, we expect that
$|g_2-1|$ is not integrable, implying  that condition \eqref{efinite}
fails. Still, we presume that the structure factor is well-defined for
wave vectors outside the reciprocal lattice of $\Z^d$. 
\end{remark}

The point processes described above were simulated in each spatial
dimension  $d \in \{1,2,3\}$ at unit intensity  $\gamma=1$,  using
periodic boundary conditions, i.e., on the flat torus.  Based on these
simulations, we computed the scattering intensity $\mathcal{S}_r(k)$, as
defined in \eqref{scattering}, for a variety of wave vectors (see
below). Since the theoretical distribution function of
$\mathcal{S}_r(k)$ is not explicitly known, we generated one million
independent realizations of each point process~---~per model, per
dimension, and for each wave vector considered.  Let
$\overline{\mathcal{S}}_r(k)$ denote the arithmetic mean of the $10^6$
realizations of $\mathcal{S}_r(k)$. We then computed the empirical
distribution function $F_{r,k}$ (say) of the normalized values
$\mathcal{S}_r(k)/\overline{\mathcal{S}}_r(k)$.  That is, for each $z
>0$, $F_{r,k}(z)$ gives the relative frequency of the $10^6$
realizations of $\mathcal{S}_r(k)/\overline{\mathcal{S}}_r(k)$ that are
less than or equal to $z$. In Figure~\ref{fig:exp_ccdf}, we plot the
{\em complementary} empirical distribution function $F_{r,k}^c := 1 -
F_{r,k}$, and, for simplicity, we denote this by  $F^c = F_{r,k}^c$.

The observation window is the cube $[0,L)^d$, where $L$ is the standard
notation for the system size in the applied literature. In our previous
notation, this corresponds to $r$. 
We use ${\mathcal{S}}(k)$ as shorthand for
$\mathcal{S}_r(k)=\mathcal{S}_L(k)$.
Throughout, we have chosen the number density to be $1$, corresponding
to an expected number of $L^d$ points
in $[0,L)^d$. Remarkably, we worked with relatively small system sizes.
More precisely, the expected number of points per sample is 1\,000 for
$d=1$, 1\,225 for $d=2$, and again 1\,000 for $d=3$.  In view of the
lattice-based models (iv) and (v), these numbers were chosen to be
integer powers of $d$. To accommodate the periodic boundary conditions,
we slightly modified the matching process (v) by performing it on the
torus. Theorem 5.1 in \cite{KLY20} strongly suggests that the difference
between this torus-based model and the restriction of the original model
to $[0,L)^d$ vanishes as $L\to\infty$. This modification requires a
corresponding adaptation of assumption \eqref{egoodpp}: namely, the
restriction of $\eta$ to $W_r$ (which determines $T_r$) must be replaced
by the set of matched points on the torus. In principle, this same
consideration applies to models (ii)-(iv). For simplicity, the data
shown in Figure~\ref{fig:exp_ccdf} are restricted to three distinct wave
vectors per dimension.  Setting $\tau := 2\pi/L$, the selected wave vectors
are:
\begin{itemize}
\item $k_1 = \tau, \ k_2 = 158 \tau, \ k_3 = 317 \tau $ if $d=1$,
\item $k_1= \tau \times (0,1), \ k_2= \tau \times (6,0), \ k_3= \tau \times (8,8)$ if $d=2$,
\item $k_1= \tau \times(0,1,0), \ k_2=\tau\times(2,2,2), \ k_3= \tau \times(3,0,0)$ if $d=3$.
\end{itemize}

\begin{figure}[t]
  \centering
  \includegraphics[width=\linewidth]{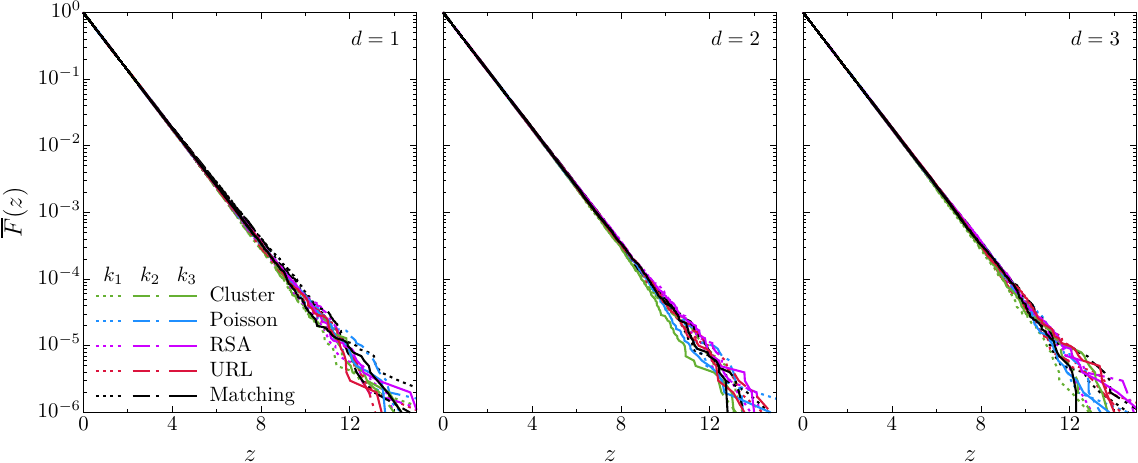}
  \caption{Empirical complementary distribution function
$F^c$ of the normalized scattering intensity
  $\mathcal{S}(k)/\overline{\mathcal{S}}(k)$.}
  \label{fig:exp_ccdf}
\end{figure}

According to Proposition~\ref{tiidexp} and the normalization
$\mathcal{S}(k) \mapsto \mathcal{S}(k)/\overline{\mathcal{S}}(k)$, the
complementary empirical distribution function $F^c$ is expected to
approximate the reciprocal exponential function, i.e., the complementary
distribution function of a unit-rate exponential distribution. In this
regard, Figure~\ref{fig:exp_ccdf} is illustrative: on a logarithmic
scale, the figure shows the empirical complementary distribution
function $F^c$ of the scattering intensity for dimensions $d=1$ (left),
$d=2$ (middle), and $d=3$ (right).  Due to the normalization, all curves
for $F^c$ collapse onto a single one, independent of the specific model
or wave vector chosen.  Moreover, in each case, the agreement with an
exponential distribution is remarkably accurate. This close
correspondence was also observed for many additional wave vectors not
shown here. Small fluctuations of $F^c(z)$ on the order of $10^{-6}$,
visible only at large values of
$\mathcal{S}(k)/\overline{\mathcal{S}}(k)$, are consistent with purely
statistical variation across our $10^6$ samples; we observe no
systematic deviations from the exponential distribution.

Finally, our results indicate only a weak correlation between
$\mathcal{S}(k)$ values for different wave vectors $k$. Even for
neighboring wave vectors  close to the origin in Fourier space, the
correlation coefficient remains consistently below 0.02. This
observation aligns perfectly with Proposition~\ref{tiidexp}.

\section{Maximum likelihood estimation}\label{secML}

In this section, we adopt an asymptotic framework naturally suggested by
Proposition~\ref{tiidexp}.  Within this setting, the background of point
processes becomes irrelevant, and our starting point is the multivariate
asymptotic distribution of scattering intensities provided by that
proposition.  This framework allows us to define maximum likelihood
estimators for a parameterized structure factor. Specifically, we employ
the expansion \eqref{skparabola} of $S(k)$ as $k\to 0$. While this
expansion introduces an additional approximation, it is justified by the
fact that \eqref{skparabola} holds in the same asymptotic regime of
large system sizes, where small wave vectors become accessible.

To proceed, consider $n$ wave vectors
$k_1,\ldots,k_n\in\R^d\setminus\{0\}$ satisfying condition
\eqref{e:0518}. Define $\kappa_j:=\|k_j\|^\alpha$ for $j\in[n]$. In a
slight abuse of prior notation, let $X_1,\ldots, X_n$ be independent
exponentially distributed random variables with expectations $\BE(X_j) =
s+t\kappa_j$ for $j\in[n]$, where $s\geq 0$ and $t\in\R$ are unknown
parameters to be estimated from observed realizations $x_1,\ldots,x_n$
of $X_1,\ldots,X_n$. This statistical problem is of interest in its own
right. Formally, the parameter space is defined as
\begin{align*}
\Theta &:=\{(s,t)\in\R^2:s\ge 0, s+t\kappa_1> 0,\ldots,s+t\kappa_n> 0\}\\
& = \bigg\{(s,t)\in\R^2:s\ge 0, t > - \frac{s}{\max(\kappa_1,\ldots,\kappa_n)} \bigg\},
\end{align*}
which describes a cone in the $(s,t)$-plane. For $(s,t)\in\Theta$, let
$\BP_{s,t}$ denote the probability measure governing the random
variables $X_1,\ldots,X_n$,  with corresponding expectation and variance
operators $\BE_{s,t}$ and $\BV_{s,t}$, respectively. This notation omits
the dependence  on $n$ and on the specific wave vectors $k_1,\ldots,k_n$
for simplicity. We formulate the null hypothesis $H_0$ as the condition
of hyperuniformity. In light of the  asymptotic regime and the
constraint from  \eqref{skparabola}, this corresponds to  the subset
$\Theta_0:=\{(s,t)\in\Theta:s=0\}$ of $\Theta$.

Under $\BP_{s,t}$, the density of $X_j$ (for each $j\in[n]$) is given by
\begin{align*}
y\mapsto  \frac{1}{s+t\kappa_j}\exp\left(-\frac{y}{s+t\kappa_j}\right), \quad y\ge 0.
\end{align*}
Given observed positive values $x_1,\ldots,x_n$, the corresponding 
log-likelihood function is
\begin{align}\label{eloglike}
\mathcal{L}(x_1,\ldots, x_n; s, t)
= \sum_{j=1}^{n}\Big[-\log (s+t\kappa_j) - \frac{x_j}{s+t\kappa_j}\Big].
\end{align}
We begin by considering the  maximum likelihood (ML) estimation of $t$
under the null hypothesis $H_0$.

\subsection{Estimation under \texorpdfstring{$H_0$}{H0}}\label{estimate H0}

Under $\BP_{0,t}$, that is, under the constraint  $s=0$ imposed by $H_0$,
the log-likelihood function simplifies to
\begin{align*}
  \mathcal{L}(x_1, \dots, x_n;0, t)=
-n \log t-\frac{1}{t}\sum_{j=1}^{n} \frac{x_j}{\kappa_j}-\sum_{j=1}^{n}\log \kappa_j.
\end{align*}
The first and second derivatives with respect to $t$ are given by
\begin{align*}
  \frac{\partial \mathcal{L}}{\partial t}(x_1, \dots, x_n;0,t)
&= -\frac{n}{t} + \frac{1}{t^2}\sum_{j=1}^{n} \frac{x_j}{\kappa_j},\\
\frac{\partial^2 \mathcal{L}}{\partial t^2}(x_1, \dots, x_n;0,t)
&=\frac{n}{t^2}-\frac{2}{t^3}\sum_{j=1}^{n} \frac{x_j}{\kappa_j}.
\end{align*}
Define the function $h_0\colon (0,\infty)^n \rightarrow \mathbb{R}$ by
\begin{align}\label{eh0}
h_0(x_1,\ldots, x_n)&:= \sup\{\mathcal{L}(x_1,\ldots,x_n;0,t):(0,t)\in\Theta\}.
\end{align}
Then the ML estimate
$\widehat{t}_0\equiv \widehat{t}_0(x_1,\ldots,x_n)$ of $t$ based on $x_1,\ldots,x_n$ is uniquely
determined by the condition 
$ 
\mathcal{L}(x_1,\ldots,x_n;0,\widehat{t}_0)=h_0(x_1,\ldots,x_n)$,
and it takes the explicit form
\begin{align} \label{eq:est_null_t}
\widehat{t}_0(x_1,\ldots,x_n)= \frac{1}{n}\sum_{j=1}^{n} \frac{x_j}{\kappa_j}.
\end{align}

Recall that the {\em gamma distribution} with
shape parameter $a$ and scale parameter $b$
has the density
\begin{align*}
y\mapsto \I\{y\ge 0\}\frac{b^{a} y^{a -1} \te^{-b x}}{\Gamma(a)}.
\end{align*}
Under $\BP_{0,t}$, the estimator $\widehat{t}_{0}(X_1,\ldots,X_n)$ is
the average of $n$ independent exponential random variables with mean
$t$. By a standard convolution property of the gamma distribution, it
follows that $\widehat{t}_{0}(X_1,\ldots,X_n)$ is gamma-distributed with
shape parameter $a = n$ and scale parameter $b = n/t$. Consequently,
\begin{align*}
\BE_{0,t} \big(\widehat{t}_{0}(X_1,\ldots,X_n)\big) = \frac{a}{b} = t, \qquad
\BV_{0,t} \big(\widehat{t}_{0}(X_1,\ldots,X_n)\big) = \frac{a}{b^2} = \frac{t^2}{n},
\end{align*}
showing that  $\widehat{t}_{0}(X_1,\ldots,X_n)$ is an unbiased and
consistent estimator of $t$. In the following, we often abuse notation
slightly by writing $\widehat{t}_{0}$ in place of
$\widehat{t}_{0}(X_1,\ldots,X_n)$; its meaning will be clear from the
context. We will adopt similar shorthand for two additional estimators
introduced below.

\begin{figure}[t]
  \centering
  \includegraphics[width=\linewidth]{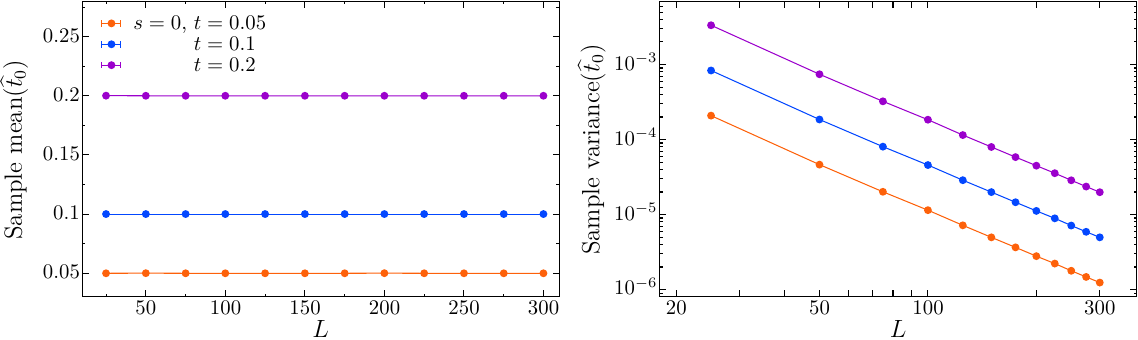}
  \caption{Unbiasedness and consistency of $\widehat{t}_0$ are
  demonstrated by plots of the sample mean value (left) and variance (right)
  as functions of the linear system size $L$.}
  \label{estimators_H0}
\end{figure}

In all our simulations, a cubic taper function is used, as is common in
physics. In accordance with Corollary~\ref{remgood}, we select the wave
vectors from the reciprocal lattice associated with a periodic
simulation box of side length $L$. Specifically, 
we choose the following \textit{commensurable} wave vectors:
\begin{align} 
  K_L := \big\{k\in \frac{2\pi}{L}\times\Z^d \setminus\{0\}: \|k\|<k_<\big\}.
  \label{eq:reciprocal}
\end{align}
In our experiments for the asymptotic setting, we take $d=2$ and
$k_<=0.75$. Other choices of $d$ and $k_<$ yield qualitatively similar
results, with differences arising primarily in the speed of convergence,
as expected.  Note that in the asymptotic setting the choice of $L$ (only)
determines which wave vectors are chosen for the simulations.
Note also that, for fixed $k_<$, the number
$n=|K_L|$ of wave vectors increases with $L$.

Figure~\ref{estimators_H0} illustrates the unbiasedness and consistency
properties of the estimator $\widehat{t}_0$ for three different values
of $t$. The left-hand panel displays the sample means as estimators
$\widehat{\BE}(\widehat{t}_0)$ of the expectation, while the right-hand
panel shows the sample variances as estimators
$\widehat{\BV}(\widehat{t}_0)$ of the variance. For each choice of the
parameters $t$ and $L$, the estimates are based on $10^6$ independent
realizations of $(X_1,\ldots,X_n)$. In the case $s=0$ and $L=100$,
Figure~\ref{histograms_H0_t0} presents the corresponding histograms
as density estimates for $\widehat{t}_0$, which align very well
with the fitted gamma distributions. Typically, histograms use vertical bars
whose areas represent frequency counts over a fixed range. In our
visualizations, however, we adopt a modified representation using
bullets instead of columns.  Each interval forming the histogram base
has equal length. The horizontal coordinates of a bullet represents the
midpoint of its base, while its height corresponds to the relative
frequency of data within that interval, divided by the base length.
Thus, the bullet heights serve as density estimates for $\widehat{t}_0$
at those midpoints. These conventions also apply, mutatis mutandis, to
all subsequent figures displaying histograms.

\begin{figure}[t]
  \centering
  \includegraphics[width=0.5\linewidth]{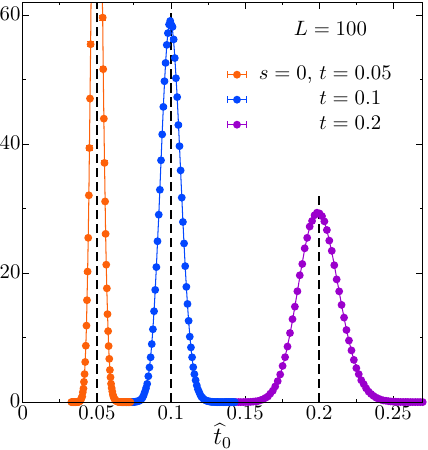}
  \caption{
    Density estimates for $\widehat{t}_0$ (bullets) and
  densities of gamma distributions with the
  same expectation and variance (solid lines).}
    \label{histograms_H0_t0}
\end{figure}

\subsection{Estimation in the full parameter space}\label{sec:fullparameter}

When $s$ and $t$ are allowed to vary freely within $\Theta$, we cannot
explicitly maximize \eqref{eloglike}. However, since $\Theta$ is a cone
in the $(s,t)$-plane, the problem of maximizing a function of two
variables can be reduced to a univariate optimization problem. To
simplify the notation, we omit the dependence on $x_1,\ldots,x_n$ in the
log-likelihood function. For $(s,t) \in  \Theta$ and $\delta >0$, we
have
\[
\mathcal{L}(\delta s, \delta t) = \mathcal{L}(s,t)  - n \log \delta  + \left(1-\frac{1}{\delta}\right) \sum_{j=1}^n \frac{x_j}{s+t\kappa_j}.
\]
Therefore, it is natural to substitute
$s=:\delta\cos\vartheta$ and $t=:\delta\sin\vartheta$, where
$\delta>0$, $\vartheta \in (\vartheta_0,\pi/2]$, 
and $\vartheta_0$ is defined by $\tan \vartheta_0 = -1/\max(\kappa_1,\ldots,\kappa_n)$.
For the moment, we temporarily exclude the case $(s,t)=(0,0)$.
Notice that
\[
\frac{\D}{\D \delta} \mathcal{L}(\delta\cos\vartheta , \delta \sin\vartheta) 
= - \frac{n}{\delta} + \frac{1}{\delta^2} \sum_{j=1}^n \frac{x_j}{\cos\vartheta+\kappa_j\sin\vartheta},
\quad \delta>0.
\]
This derivative vanishes if and only if
\[
\delta = \delta_0(\vartheta)
:= \frac{1}{n} \sum_{j=1}^n \frac{x_j}{\cos\vartheta+\kappa_j\sin\vartheta}.
\]
Assume that $\delta_0(\vartheta)>0$.
Since $\frac{\D}{\D \delta} \mathcal{L}(\delta \cos\vartheta, \delta \sin\vartheta)$ is
positive for $\delta < \delta_0(\vartheta)$ and negative for $\delta > \delta_0(\vartheta)$,
the function
$(0,\infty) \ni \delta \mapsto \mathcal{L}(\delta \cos\vartheta, \delta \sin\vartheta)$ 
attains a unique maximum at $\delta = \delta_0(\vartheta)$.  
Defining
\begin{align}\label{e:L*}
  \mathcal{L}^*(\vartheta) := \max_{\delta \ge 0} \mathcal{L}(\delta \cos \vartheta, \delta \sin \vartheta),
\end{align}
and  substituting
$\delta_0(\vartheta)$ into the definition of $\mathcal{L}(\delta s, \delta t)$,
we obtain 
\begin{align}\label{e:L*2}
  \mathcal{L}^*(\vartheta)=\mathcal{L}(\delta_0(\vartheta)\cos \vartheta,\delta_0(\vartheta)\sin \vartheta).
\end{align}
This remains true if $\delta_0(\vartheta)=0$ or, equivalently, if $x_1=\cdots=x_n=0$.
This case is of no practical relevance and will be excluded in the sequel.
We will maximize $\mathcal{L}$ by first maximizing $\mathcal{L}^*$.
Some properties of $\mathcal{L}^*$ are discussed in Appendix~\ref{appendix3}.

Similarly to \eqref{eh0}, we define a function $h_1$ on $(0,\infty)^n$ by
\begin{align}\label{eh1}
h_1(x_1, \ldots, x_n) &:= \sup\{\mathcal{L}(x_1,\ldots,x_n;s,t):(s,t)\in\Theta\}.
\end{align}
Given $(x_1,\ldots,x_n)$, let $\widehat{\Theta}(x_1, \ldots, x_n)$ denote
the set of all $(s',t')\in\Theta$ such that
\begin{align*}
\mathcal{L}(x_1,\ldots,x_n;s',t')=h_1(x_1, \ldots, x_n).
\end{align*}
By Lemma \ref{lemma1lq}, the set $\widehat{\Theta}(x_1, \ldots, x_n)$
is non-empty.

\begin{remark}\label{runique}\rm Let $(x_1,\ldots,x_n)\in (0,\infty)^n$.
Our numerical analysis and Lemma~\ref{lemma1lq}  strongly suggest that
$\mathcal{L}^*$ has a unique maximizer $\widehat{\vartheta}(x_1,\ldots,x_n)$, 
 that is
\begin{align}\label{eunique}
\big\{\vartheta':\mathcal{L}^* (\vartheta')=\max_{\vartheta_0< \vartheta\le \pi/2} \mathcal{L}^*(\vartheta)\big\}
 =\{\widehat{\vartheta}(x_1,\ldots,x_n)\}.
\end{align}
Then the original likelihood function $\mathcal{L}$ has a unique 
maximizing pair $(\widehat{s},\widehat{t}_1)$ (depending on the data)
given by 
\begin{align}
  \label{eq:s_Lstar}
  \widehat{s} &= \widehat{\delta}\cos\widehat{\vartheta}, \\
  \widehat{t}_1 &= \widehat{\delta}\sin\widehat{\vartheta},
  \label{eq:t1_Lstar}
\end{align}
where
\begin{align*}
\widehat{\delta}:=\delta_0(\widehat\vartheta)=
\frac{1}{n}\sum_{j=1}^n \frac{x_j}{\cos\widehat{\vartheta} +\kappa_j\sin\widehat{\vartheta}}.
\end{align*}
Whenever we use the notation
$\widehat{s}(\cdot)$ and $\widehat{t}_1(\cdot)$,
we implicitly assume  \eqref{eunique}.
\end{remark}

Before discussing the findings of our simulations, we present
two useful invariance properties.

\begin{lemma}\label{linv} Let $t>0$. Then
\begin{align*}
\BP_{0,t}(h_1(X_1,&\ldots,X_n)-h_0(X_1,\ldots,X_n)\in\cdot)\\
&=\BP_{0,1}(h_1(X_1,\ldots,X_n)-h_0(X_1,\ldots,X_n)\in\cdot).
\end{align*}
\end{lemma}
\begin{proof} The parameter space $\Theta$ is invariant under positive scaling.
Furthermore, the log-likelihood function
is essentially scale-invariant, since
\begin{align*}  
\mathcal{L}(r x_1, \ldots, r x_n; r s, r t) = \mathcal{L}(x_1, \ldots, x_n; s, t) - n \log r,
\end{align*}
for any $r>0$, any $(s,t)\in\Theta$, and any
$(x_1, \ldots, x_n)\in (0,\infty)^n$. It follows that
\begin{align*}
h_1(rx_1,\ldots,rx_n)-h_0(rx_1,\ldots,rx_n)=h_1(x_1,\ldots,x_n)-h_0(x_1,\ldots,x_n).
\end{align*}
Now, set $r:=t^{-1}$, and note the distributional
identity
\begin{align*}
\BP_{0,t}\big( (t^{-1}X_1,\ldots,t^{-1}X_n)\in\cdot\big)=\BP_{0,1}\big( (X_1,\ldots,X_n)\in\cdot\big),
\end{align*}
which completes the proof.
\end{proof}

If assumption \eqref{eunique} holds, the following
 corollary implies that
\begin{align*}
\BP_{0,t}\big( \widehat{s}(X_1,\ldots,X_n)=0\big) = \BP_{0,1}\big( \widehat{s}(X_1,\ldots,X_n)=0\big),
\quad t>0.
\end{align*}

\begin{figure}[t]
  \centering
  \includegraphics[width=\linewidth]{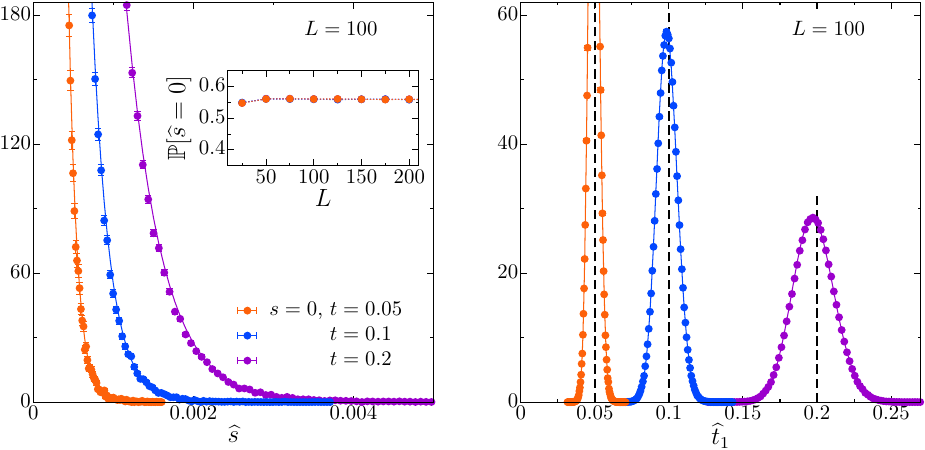}
  \caption{
    Density estimates for the continuous part of $\widehat{s}$ and for $\widehat{t}_1$ (bullets) and
  densities of gamma distributions with the same expectation and
  variance (solid lines). The inset shows the size of the atom of
  $\widehat{s}$ as a function of $L$.}
  \label{histograms_H0_t1}
\end{figure}

\begin{figure}[t]
  \centering
  \includegraphics[width=\linewidth]{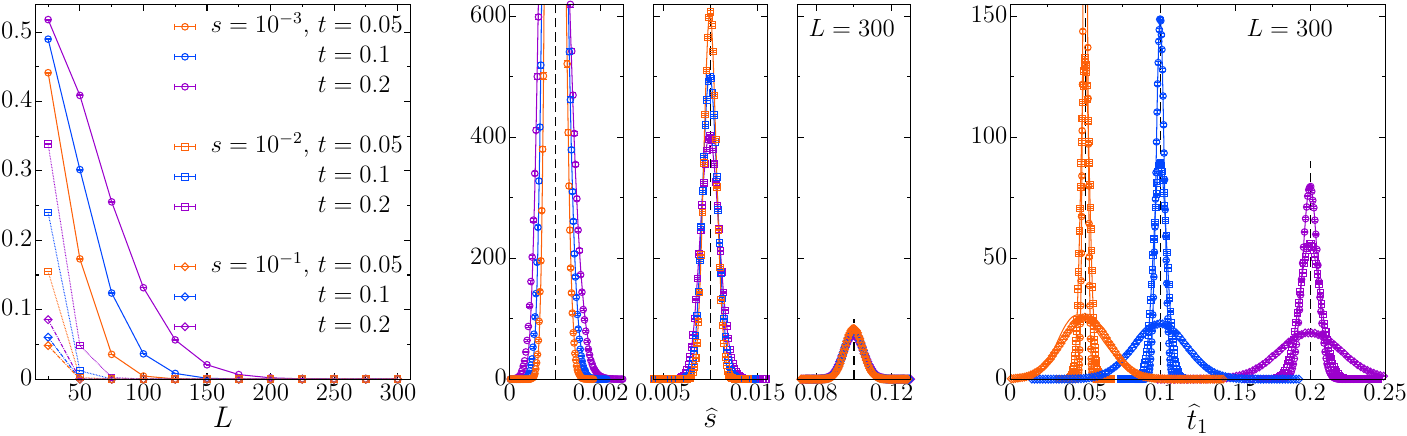}
  \caption{
  The non-hyperuniform case: (left) the size of the atom of
  $\widehat{s}$; (center and right) density estimates for
  $\widehat{s}$  and $\widehat{t}_1$ (bullets) and densities of gamma
  distributions with the same expectation and variance (solid lines).}
  \label{histograms_H1}
\end{figure}

\begin{corollary}\label{cinv} Let $t>0$. Then
\begin{align*}
\BP_{0,t}\big( \widehat{\Theta}(X_1,\ldots,X_n)\cap\Theta_0=\emptyset\big)
=\BP_{0,1} \big(\widehat{\Theta}(X_1,\ldots,X_n)\cap\Theta_0=\emptyset \big).
\end{align*}
\end{corollary}
\begin{proof} Let $(x_1, \ldots, x_n)\in(0,\infty)^n$, and suppose that 
$\widehat{\Theta}(x_1,\ldots,x_n)\cap\Theta_0\ne\emptyset$.
Then, by definition, we immediately obtain 
$h_0(x_1,\ldots,x_n)=h_1(x_1,\ldots,x_n)$. Conversely, assume that
the above identity holds. As shown earlier, there exists
some $\widehat{t}_0>0$ such that $\mathcal{L}(x_1,\ldots,x_n;0,\widehat{t}_0)=h_0(x_1,\ldots,x_n)$.
Hence, $(0,\widehat{t}_0)\in \widehat{\Theta}(x_1,\ldots,x_n)$, implying that
$\widehat{\Theta}(x_1,\ldots,x_n)\cap\Theta_0\ne\emptyset$.
The assertion then follows directly from Lemma \ref{linv}.
\end{proof}

We numerically optimize $\mathcal{L}(x_1,\ldots,x_n;s,t)$
to machine precision, using a tolerance of about $10^{-14}$.
Such high accuracy is essential for reliably studying the
$\BP_{s,t}$-distribution of $\widehat{s}(X_1,\ldots,X_n)$,
especially since the density of its continuous part diverges as $s\to 0$,
as suggested by simulations.

\begin{figure}[t]
  \centering
  \includegraphics[width=\linewidth]{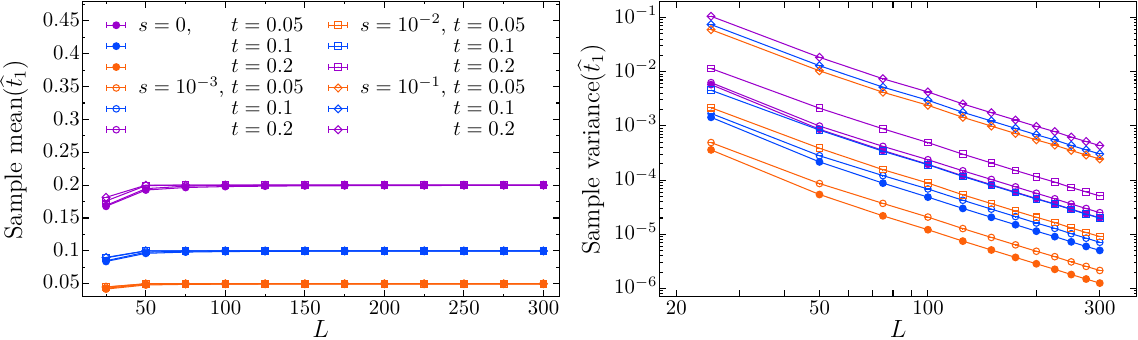}
  \caption{Asymptotic unbiasedness and consistency of $\widehat{t}_1$
   are illustrated by plots of the sample mean value (left) and variance (right)
  as functions of the linear system size $L$.}
  \label{estimators_H1_t}
\end{figure} 

\begin{figure}[t]
  \centering
  \includegraphics[width=\linewidth]{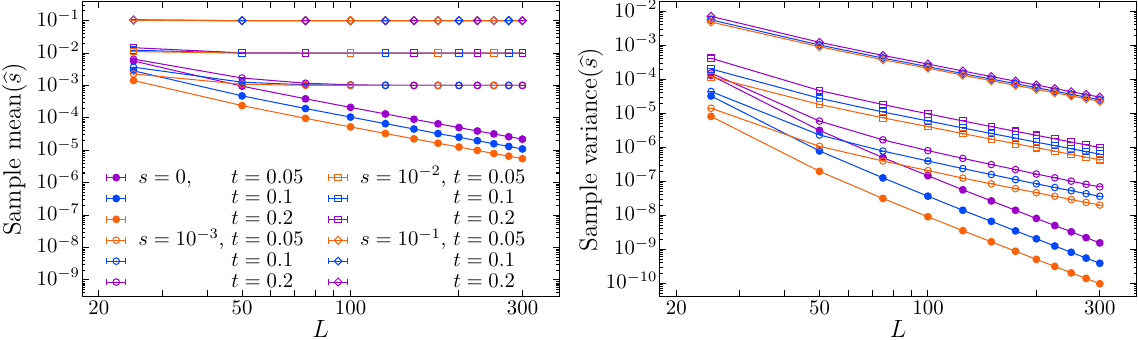}
  \caption{Asymptotic unbiasedness and consistency of $\widehat{s}$
  are demonstrated by plots of the mean value (left) and variance (right)
  as functions of $L$.}
  \label{estimators_H1_s}
\end{figure}

Using the commensurable wave vectors $K_r$ from
\eqref{eq:reciprocal}, i.e., the same wave vectors as for Figures~\ref{estimators_H0}
and \ref{histograms_H0_t0}, we repeat the optimization for $10^6$
independent simulations of the vector of $(X_1,\ldots,X_n)$ for each
choice of $s$, $t$, and $L$. We then approximate the densities of
$\widehat{s}(X_1,\ldots,X_n)$ and $\widehat{t}_1(X_1,\ldots,X_n)$ using
histograms of the simulation results. See Figures~\ref{histograms_H0_t1}
and \ref{histograms_H1} for simulations with $s=0$ and $s>0$,
respectively. The inset in Figure~\ref{histograms_H0_t1} and the left
panel in Figure~\ref{histograms_H1} illustrate the stabilization of the
atom's mass. Asymptotic unbiasedness and consistency are demonstrated in
Figures~\ref{estimators_H1_t} and \ref{estimators_H1_s}.

{Our simulations strongly suggest the following asymptotic behavior
of our estimators:
\begin{itemize}
\item The estimators are consistent and asymptotically
    unbiased.
\item The asymptotic distribution  of
    $\widehat{s}(X_1,\ldots,X_n)$, as $L \to \infty$, includes an atom at  the origin.
    The atom mass $\BP_{s,t}(\widehat{s}(X_1,\ldots,X_n)=0)$
    converges to zero if $s>0$, and to a constant
   (independent of $t$) if $s=0$.
    This $t$-independence follows from Corollary \ref{cinv}.
\item
 As $L\to\infty$, the continuous parts of the distributions of both
    $\widehat{s}(X_1,\ldots,X_n)$ and $\widehat{t}_1(X_1,\ldots,X_n)$
    converge to  gamma distributions.
For $s=0$, these gamma approximations are already accurate at relatively small system sizes.
\end{itemize}

\section{Likelihood ratio test for hyperuniformity}\label{seclratio}

\begin{figure}[t]
  \centering
  \includegraphics[width=\linewidth]{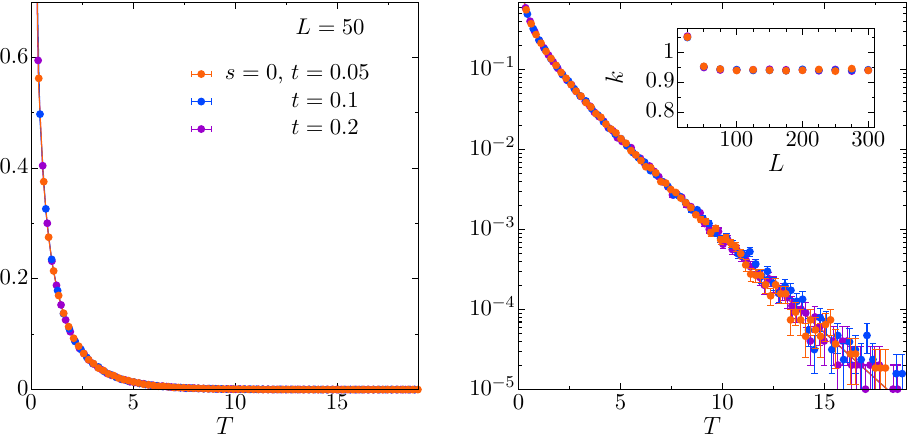}
  \caption{Density estimates for the continuous part of $T$ 
    (bullets) in our asymptotic framework, i.e., for exponentially
      distributed and independent values of $S(k)$ with $s=0$, and
      densities of a $\chi^2$-distribution with the same expectation
      (solid lines). A linear plot is shown on the left-hand side, and a
      log-log plot on the right-hand side. The inset shows the {\em
      fractional} parameter $k$
      of the $\chi^2$-distribution as a function of $L$.}
  \label{fig:exp}
\end{figure}

In this section, we adopt the framework established in Section \ref{secML}.
We define the test statistic
\begin{align}
  T(x_1, \dots, x_n) := 2[h_1(x_1, \dots, x_n) - h_0(x_1, \dots, x_n)], \quad
(x_1,\ldots,x_n)\in(0,\infty)^n,
\end{align}
where $h_0$ and $h_1$ are defined in \eqref{eh0} and \eqref{eh1}, respectively.
By construction, $T$ is nonnegative.
Under our hypothesis \eqref{eunique}, we have $T(x_1,\ldots,x_n)=0$
if and only if $\widehat{s}(x_1,\ldots,x_n)=0$. Hence,
the $\BP_{s,t}$-atom of $\widehat{s}(X_1,\ldots,X_n)$
carries over to the distribution of $T(X_1,\ldots,X_n)$.

\begin{figure}[t]
  \centering
  \includegraphics[width=\linewidth]{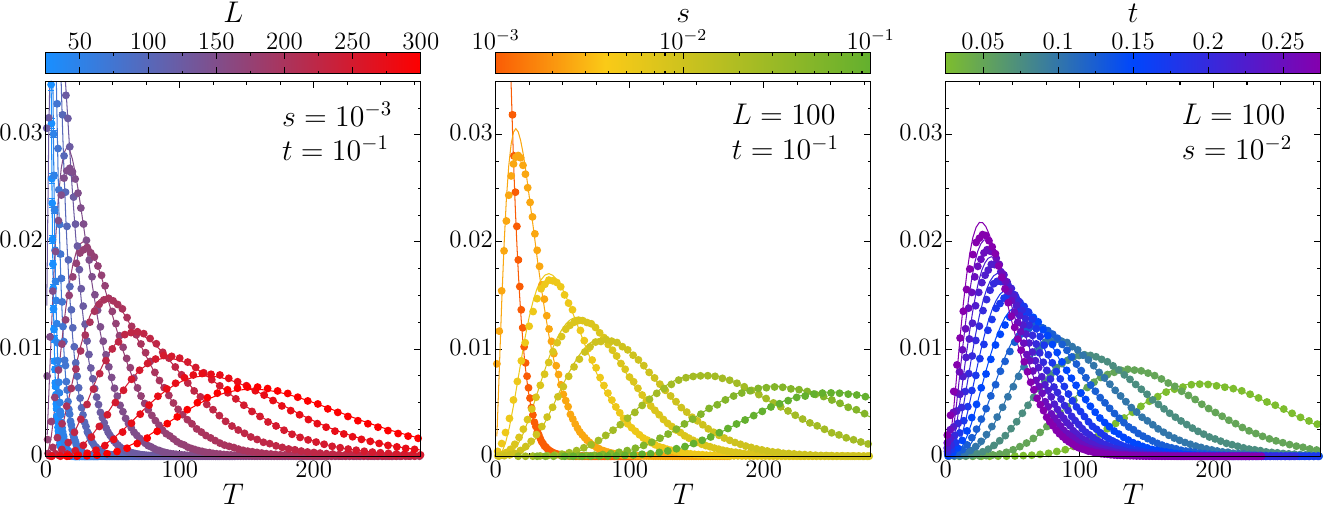}
  \caption{Density estimates for the continuous part of $T$ in our
    asymptotic framework, but now with
  different values of $s$, $t$, and $L$.}
  \label{fig:nonhyp}
\end{figure}

Figure~\ref{fig:exp} presents the simulations results for the test
statistic $T$, specifically showing the atom size and density estimates for the
continuous part of $T$. The same set of samples is used to compute $T$ for
Fig.~\ref{fig:exp} and $(\widehat{s},\widehat{t}_1)$ for
Fig.~\ref{histograms_H0_t1}.

Based on the simulation results shown in Fig.~\ref{fig:exp}, we
conjecture that, as $n \to \infty$, the limiting distribution of $T$
under $H_0$ is a mixture of an atom at $0$ and a gamma distribution with
parameters $a=k/2$ and $b=1/2$, where $k$ is slightly less than $1$. For
$k=1$, this would correspond to a $\chi^2$-distribution with one degree
of freedom. In our case, we refer to this limiting form as a
$\chi^2$-distribution with  {\em fractional} parameter $k$.
Therefore, we assume in the following that the limiting
distribution of $T$ consists of an atom of size 0.5585 and a continuous
part with fractional parameter $k=0.9400$. Importantly, this
approximation holds well even for small system sizes, and it
implies that the asymptotic distribution of $T$ does not depend on $t$.
In fact, this $t$-independence follows directly from Lemma \ref{linv}.
The existence of an universal (i.e., $t$-independent) explicit limit
distribution for $T$ is crucial for our statistical test of
hyperuniformity.

Under $H_1$, the distribution of $T$ is asymptotically (as
$n\to\infty$) well approximated by a gamma distribution; see
Figure~\ref{fig:nonhyp}, which is based on the same simulation settings as in
Section~\ref{sec:fullparameter} but extended to additional values of
$s$, $t$, and $L$. Moreover, we find that the distribution of $T$ can
be well approximated using only a single parameter, i.e., a scatter plot
of $a$ and $b$ collapses onto a single curve.

\section{Testing hyperuniformity in an example}
\label{sectestexample}

We conclude by demonstrating  the practical relevance and remarkable
sensitivity of the novel test through its application to model (v) of matched
points, as introduced  in Section~\ref{asymptsctt}. This model is
particularly advantageous: although it exhibits non-trivial
correlations, it is known to be hyperuniform~\cite{KLY20}. Furthermore,
it allows for efficient simulation of large samples, and the parameter
$t$  can be conveniently tuned by adjusting the intensity $\rho$. While
our simulations focus on two-dimensional systems, the results are, in
principle, applicable in any spatial dimension.

We first verify that the approximations developed in previous sections
remain highly accurate for realistically sized  systems commonly
encountered in applications; see~\cite{Torquato2018} and references
therein. Figure~\ref{fig:pkt} displays the empirical probability density
of the test statistic $T$ for various values of $t$. In all  cases, the
empirical distributions collapse onto the predicted
$\chi^2$-distributions,  within statistical error margins. The insets of
Figure~\ref{fig:pkt} plot the two defining parameters of the limiting
distribution, namely, the atom size and the `fractional degree of
freedom', as functions of the linear system size $L$. The data obtained
from simulations (indicated by bullets) closely match the asymptotic
predictions (represented by horizontal lines).

The specific simulation parameters are as follows. The maximum length of
the wave vectors is adapted to the value of $\rho$, ensuring that the
parabolic approximation remains valid with sufficient accuracy (see the
following two paragraphs for discussion). Specifically, the parameter
$k_<$ in equation \eqref{eq:reciprocal} is set to
0.50 for $\rho=3.0$ (corresponding to $t\approx 0.05$),
0.40 for $\rho=2.5$ ($t\approx 0.07$), and
0.33 for $\rho=2.0$ ($t\approx 0.09$).
These are heuristic choices, and our results are not sensitive to
moderate variations in $k_<$. The test statistic $T$ is computed
independently for each sample. The number of independent realizations
per system size is:
1\,000\,000 for $L=50$,
400\,000 for $L=75$,
200\,000 for $L=100$,
100\,000  for $L=125$,
80\,000  for $L=150$,
60\,000  for $L=175$, and
40\,000  for $L=200$.

An general challenge in any numerical study of hyperuniformity lies in
ensuring that sample sizes are large enough to reveal the asymptotic
scaling behavior of the structure factor (or equivalently, the number
variance). For the matching process in particular, the onset of the
hyperuniform scaling regime depends on the value of $\rho$. As $\rho$
approaches one, the range of wavenumbers over which the hyperuniform
scaling $S(k)\sim k^2$ holds diminishes; see Figure~9~in \cite{KLY20}
and Figure~10 for a corresponding analysis based on the number variance.
Once the system size becomes sufficiently large, we observe a class~I
hyperuniform scaling with $\alpha=2$, consistent with Remark~8.3~in
\cite{KLY20}.

\begin{figure}[t]
  \centering
  \includegraphics[width=\linewidth]{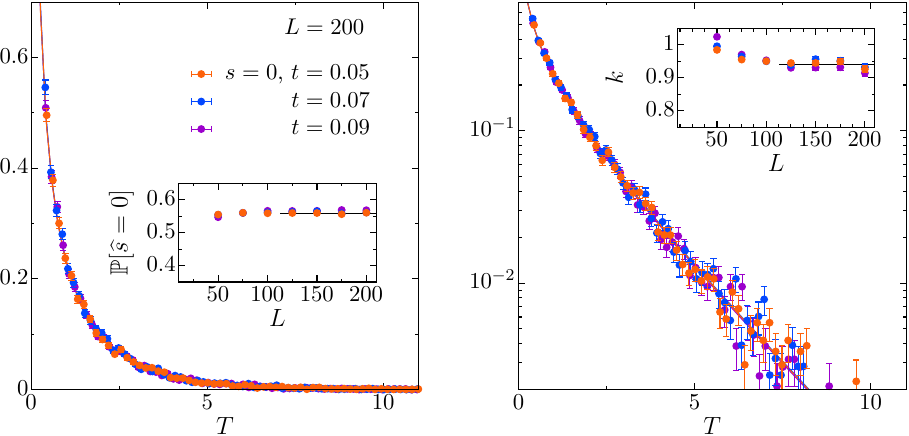}
  \caption{Density estimates for the continuous part of $T$
    (bullets) for matched point patterns, and densities of a
    $\chi^2$-distribution with the same expectation (solid lines).
    A linear plot is shown on the left-hand side, and a log-log plot
    on the right-hand side.
    The insets show the size of the atom of $T$ (left) and
    the {\em fractional} parameter $k$ of the $\chi^2$-distribution as a
    function of $L$. The solid lines in the inset indicate the
    universally chosen parameters for the test, i.e., an atom of size
    0.5585 and a fractional parameter $k=0.9400$.}
  \label{fig:pkt}
\end{figure}

It is important to note that our test remains valid even for smaller
system sizes or larger wavenumber ranges where the structure factor
deviates from the parabolic approximation. In such cases, rejection of
the null hypothesis may occur solely due to the incorrect functional
form of $S(k)$. For this reason, the test should also correctly reject
class II and class III hyperuniform systems, which are characterized by
different asymptotic scaling laws~\cite{Torquato2018}. This limitation
exemplifies a broader statistical issue; rejection of the null
hypothesis does not necessarily imply non-hyperuniformity, but may
rather indicate a departure from a specific scaling assumption. 

Despite these caveats, we find our test to be robust against small
deviations from the approximation \eqref{skparabola}. In fact,
\eqref{skparabola} may not strictly hold for the matching process, since
the underlying lattice introduces anisotropy. Nevertheless, the
simulations strongly suggest that $S(k)$ is effectively isotropic,
except at the wave vectors of the reciprocal lattice of $\Z^d$. As a
result, the distribution of $T$ is indistinguishable from the asymptotic
distribution under the null hypothesis $H_0$, as illustrated in
Figure~\ref{fig:pkt}.

Next, we employ independent thinning to continuously vary the value of
$S(0)$. Given $p\in[0,1]$ and a stationary point process $\eta$, the
$p$-thinning $\eta_p$ of $\eta$ is defined by retaining each point of
$\eta$ independently of each other with probability $p$; see
\cite[Section 5.3]{LastPenrose17}. A straightforward computation shows
that the intensity $\gamma_p$ and the reduced covariance measure
$\beta^!_{2,p}$ of $\eta_p$ are given by $\gamma_p=p\gamma$ and
$\beta^!_{2,p}=p^2\beta^!_2$, respectively. Consequently, if the
structure factor is defined via equation \eqref{SFgeneral}, we obtain
\begin{align}\label{SFthinning}
S_p(k)=1-p +p S(k),\quad k\in\R^d.
\end{align}
This identity holds for all $k$ for which $S(k)$ (and hence $S_p(k)$) is
well defined. 

By thinning we gain fine control over the value of $s:=S(0)$, the
structure factor at the origin. We begin with $s=10^{-4}$, which is an
order of magnitude smaller than typical values encountered in
non-hyperuniform models~\cite{Torquato2018, klatt_universal_2019}, and
increase $s$ up to $0.1$. For each combination of $s$ and system size
$L$, we generate 10\,000 independent realizations of the point process.

\begin{table}
  \renewcommand{\arraystretch}{1.15}
  \centering
  \caption{Power of the novel test for hyperuniformity
  for a range of values of $s$ and $L$. Each test is applied to only a single
  sample of the thinned model (v).}\label{tab:power}
  \begin{tabular}{l|*{7}{c}}
  \hline
  \multicolumn{1}{c|}{\multirow{2}{*}{$s$}} & \multicolumn{7}{c}{$L$} \\
          &       50 &      100 &      150 &      200 &      250 &      300 &      400 \\
  \hline
  0.0     &     0.05 &     0.05 &     0.06 &     0.05 &     0.05 &     0.05 &     0.05 \\
  0.0001  &     0.08 &     0.20 &     0.41 &     0.63 &     0.82 &     0.93 &     0.99 \\
  0.00025 &     0.13 &     0.41 &     0.77 &     0.94 &     0.99 &     1.00 &   $\ast$ \\
  0.0005  &     0.20 &     0.69 &     0.96 &     1.00 &   $\ast$ &   $\ast$ &   $\ast$ \\
  0.00075 &     0.27 &     0.84 &     0.99 &     1.00 &   $\ast$ &   $\ast$ &   $\ast$ \\
  0.001   &     0.34 &     0.92 &     1.00 &   $\ast$ &   $\ast$ &   $\ast$ &   $\ast$ \\
  0.0025  &     0.67 &     1.00 &   $\ast$ &   $\ast$ &   $\ast$ &   $\ast$ &   $\ast$ \\
  0.005   &     0.88 &   $\ast$ &   $\ast$ &   $\ast$ &   $\ast$ &   $\ast$ &   $\ast$ \\
  0.0075  &     0.94 &   $\ast$ &   $\ast$ &   $\ast$ &   $\ast$ &   $\ast$ &   $\ast$ \\
  0.01    &     0.97 &   $\ast$ &   $\ast$ &   $\ast$ &   $\ast$ &   $\ast$ &   $\ast$ \\
  0.025   &     0.99 &   $\ast$ &   $\ast$ &   $\ast$ &   $\ast$ &   $\ast$ &   $\ast$ \\
  0.05    &     1.00 &   $\ast$ &   $\ast$ &   $\ast$ &   $\ast$ &   $\ast$ &   $\ast$ \\
  0.075   &     1.00 &   $\ast$ &   $\ast$ &   $\ast$ &   $\ast$ &   $\ast$ &   $\ast$ \\
  0.1     &     1.00 &   $\ast$ &   $\ast$ &   $\ast$ &   $\ast$ &   $\ast$ &   $\ast$ \\
  \hline
  \end{tabular}
\end{table}

Table~\ref{tab:power} reports the fraction of samples for which the null
hypothesis is rejected, rounded to two decimal places. These values
provide an empirical measure of the power of the test. Crucially, recall
that the test statistic $T$ is evaluated separately for each
realization, i.e., for single point patterns. An asterisk indicates that
every sample in the corresponding setting leads to  rejection,
signifying power one. The nominal significance level is fixed at $0.05$.

For all simulations, we set $\rho = 3.0$, corresponding to $t\approx
0.05$. The parameter $k_<$, determining the maximum wave vector
length in equation \eqref{eq:reciprocal}, is set to $0.75$, as
before. We observed that modest increases in $k_<$ can enhance the
power of the test while still preserving the significance level.
However, if $k_<$ is chosen too large, the parabolic approximation
underlying the test becomes invalid. 

The first row of Table~\ref{tab:power} confirms that the test maintains
the nominal significance level under the null hypothesis,  within
statistical fluctuations. For any $s>0$, the power increases rapidly
with either the system size $L$ or the magnitude of the structure factor
$s$ at the origin. Notably, even with relatively small samples
containing just 2\,500 points ($L=50$), nearly all realizations are
rejected when $s>10^{-2}$. For the challenging case $s=10^{-4}$, often
regarded as a borderline scenario in physical
applications~\cite{Torquato2018, klatt_universal_2019}, the test still
rejects over $90\%$ of samples at $L=300$, corresponding to 90\,000
points.

\section{Summary of test procedure}

The novel test procedure is summarized below (for details, see the
open‑source implementation and its documentation~\cite{Klatt2025}). As
with any null hypothesis test, the first step is to specify a
significance level, denoted $\zeta$, e.g., 5\%.  

Next, one selects the hyperuniformity scaling exponent $\alpha$ and a
taper function $h$, together with the corresponding wave vectors, based
on the data, such as assumptions about the underlying model or
measurement details. The default choice for $\alpha$ is two: values
smaller than two require long-range correlations, whereas larger values
demand a mechanism that suppresses the leading-order term.  For
simulated data with periodic boundary conditions, a standard choice of
$h$ is the cubic taper function described in Example~\ref{excube},
paired with the commensurable wave vectors $K_L$ defined in
\eqref{eq:reciprocal}.

The maximum wavenumber $k_{<}$ should be chosen at the onset of the
so-called scaling regime of the scattering intensity, i.e., where the
approximation \eqref{skparabola} becomes valid. This value can be read
directly from a log--log plot of $\mathcal{S}(k)$. The test is robust to
the precise choice of this parameter; however, a value that is too small
amplifies statistical fluctuations, while choosing it too large
introduces systematic errors since approximation \eqref{skparabola}
is then violated.

Our implementation~\cite{Klatt2025} proceeds as follows.
\begin{enumerate}
  \item The critical value $T_{c}$ of the test statistic is computed as
    \begin{align*}
      T_{c} := F^{-1}_{\chi^{2},k}\left(
        \frac{1-\zeta-\BP_{s,t}(T=0)}{1-\BP_{s,t}(T=0)}
      \right),
    \end{align*}
    where $F^{-1}_{\chi^{2},k}$ denotes the inverse
    $\chi^{2}$-distribution with a fractional parameter $k=0.9400$, and
    the atom of the distribution of $T$, $\mathbb{P}_{s,t}(T=0)=0.5585$,
    is taken into account. For a significance level of $5\%$, this
    yields  $T_{c}\approx 2.382$.

  \item In preparation of the actual test, the data are loaded, the
    parameters are read, and the wave vectors and corresponding
    scattering intensities are computed.

  \item Under the null hypothesis, $\widehat{t}_{0}$ is obtained from
    \eqref{eq:est_null_t}, and the corresponding $h_{0}$ is evaluated.

  \item In the full parameter space, the log-likelihood
    $\mathcal{L}^{*}(\vartheta)$ from \eqref{e:L*2} is maximized,
    yielding $\widehat{s}$, $\widehat{t}_{1}$, and $h_{1}$ as defined in
    \eqref{eh1}, \eqref{eq:s_Lstar}, and \eqref{eq:t1_Lstar}.

  \item The test statistic $T=2(h_{1}-h_{0})$ is evaluated.
\end{enumerate}
The null hypothesis is accepted if $T<T_{c}$ and rejected otherwise.

\section{Final remarks}

Our test relies on three key approximations, each of which becomes valid
in the same thermodynamic limit, that is, for large system sizes. First,
in this limit, the scattering intensity becomes both exponentially
distributed and independent across different wave vectors. Second, the
number of available wave vectors in a neighborhood around the origin
diverges, and the distribution of the test statistic converges to a
specific mixture: a point mass at zero and a gamma distribution. Future
work may establish that this  convergence also holds for finite (but
sufficiently large) sample sizes, provided the number of samples
increases accordingly. Third, in the thermodynamic limit, the
approximation \eqref{skparabola} holds in the following sense: as the
system size grows, wave vectors of smaller magnitude become accessible,
thereby improving the approximation. Accordingly, the range of wave
numbers used for fitting must be adjusted as the sample size increases.

An intriguing direction for future research is to expand the class of
point processes that satisfy the core assumption \eqref{egoodpp}; see
Remark \ref{remgood2}. Additional studies could also provide further
insight into the asymptotic behavior of the maximum likelihood estimator
and the test statistic, as suggested by the simulation results. To our
knowledge, there are no existing results in the literature that directly
address sequences  of independent exponential random variables with
varying means governed by constrained parameters.

We have applied the novel test in extensive simulations to thinnings of the
matching process from \cite{KLY20}. Preliminary results for other point
processes show similarly promising behavior. Finally, we will release an
open-source Python package to support the application of our test to a
wide variety of datasets. We greatly appreciate any feedback and
suggestions.

\section{Appendices}

\subsection{Proof of Proposition \ref{pcov} (asymptotic covariances)} \label{appendix1}

\begin{proof}
We first assume that $k\ne \mathbf{0}$ and $\ell\ne \mathbf{0}$. Without loss of generality, we can
then assume that $k_r\ell_r\ne 0$ for all $r>0$.  Hence, we have
\begin{align*}
\BE T_r(k_r)\overline{T_r(\ell_r)}
&= \frac{1}{r^d}\BE \bigg[\sum_{x,y\in\eta}h_r(x)h_r(y)\te^{-\ti\langle k_r,x\rangle}\te^{\ti\langle \ell_r,y\rangle}\bigg]\\
&= \frac{1}{r^d} \BE \bigg[\sum_{x\in\eta}h_r(x)^2\te^{-\ti\langle k_r,x\rangle}\te^{\ti\langle \ell_r,x\rangle}\bigg]\\
&\quad+ \frac{1}{r^d} \BE\bigg[ \sideset{}{^{\ne}}\sum_{x,y\in\eta}
h_r(x)h_r(y)\te^{-\ti\langle k_r,x\rangle}\te^{\ti\langle \ell_r,y\rangle}\bigg].
\end{align*}
By applying Campbell's formula \eqref{eCampbell} along with \eqref{edisint}, this leads to
\begin{align}\label{e389}\notag
\BE T_r(k_r)\overline{T_r(\ell_r)}&=
\frac{\gamma}{r^d} \int h_r(x)^2\te^{\ti\langle \ell_r-k_r,x\rangle}\,\D x\\
&\quad+ \frac{1}{r^d} \iint h_r(x)h_r(x+y)
\te^{-\ti\langle k_r,x\rangle}\te^{\ti\langle \ell_r,x+y\rangle}\,\D x\,\alpha^!_2(\D y)\\
\notag
& =:C_1(r)+C_2(r),
\end{align}
say. By \eqref{e977}, we obtain
\begin{align*}
\lim_{r\to\infty}C_1(r)=\gamma M_2(h)\I\{k=\ell\}.
\end{align*}
From assumption \eqref{e234}, we also have
\begin{align*}
\lim_{r\to\infty}C_2(r)=\lim_{r\to\infty}C'_2(r),
\end{align*}
where
\begin{align*}
C'_2(r):= \frac{1}{r^d} \iint h_r(x)h_r(x+y)\te^{-\ti\langle k_r,x\rangle}
\te^{\ti\langle \ell_r,x+y\rangle}\,\D x\,\beta^!_2(\D y).
\end{align*}
Rewriting $C'_2(r)$, we obtain 
\begin{align*}
C'_2(r)&= \frac{1}{r^d} \iint h_r(x)h_r(x+y)
\te^{\ti\langle \ell_r-k_r,x\rangle}\te^{\ti\langle \ell_r,y\rangle}\,\D x\,\beta^!_2(\D y)\\
&= \frac{1}{r^d}\iint h_r(x)^2
\te^{\ti\langle \ell_r-k_r,x\rangle}\te^{\ti\langle \ell_r,y\rangle}\,\D x\,\beta^!_2(\D y)\\
&\qquad+ \frac{1}{r^d} \iint h_r(x)(h_r(x+y)-h_r(x))
\te^{\ti\langle \ell_r-k_r,x\rangle}\te^{\ti\langle \ell_r,y\rangle}\,\D x\,\beta^!_2(\D y)\\
& =:D_1(r)+D_2(r),
\end{align*}
say. Furthermore,
\begin{align*}
D_1(r)= \frac{1}{r^d} \int h_r(x)^2\te^{\ti\langle \ell_r-k_r,x\rangle}\,\D x
\int \te^{i\langle \ell_r,y\rangle}\,\beta^!_2(\D y).
\end{align*}
If $(k_r)=(\ell_r)$, then 
\begin{align*}
D_1(r)= M_2(h) \int \te^{i\langle \ell_r,y\rangle}\,\beta^!_2(\D y).
\end{align*}
By assumption \eqref{efinite3} and the dominated convergence theorem,
\begin{align*}
\lim_{r\to\infty}D_1(r)=M_2(h)\gamma(S(k)-1).
\end{align*}
From \eqref{e977}, in all other cases, we obtain $D_1(r)\to 0$ as $r\to\infty$.

We now show that $D_2(r)\to 0$ as $r\to\infty$.
Let $c>0$ be an upper bound for $|h|$.
Then, using \eqref{e239}, we obtain
\begin{align*}
|D_2(r)|\le \frac{c}{r^d} &\iint |h_r(x+y)-h_r(x)|\,\D x\,|\beta^!_2|(\D y)\\
&=c\iint |h(x+r^{-1}y)-h(x)|\,\D x\,|\beta^!_2|(\D y)\\
&\le ca\int r^{-p}\|y\|^p\,|\beta^!_2|(\D y)
+2M_1(h)\int \I\{\|y\|> rb\}\,|\beta^!_2|(\D y). 
\end{align*}
By \eqref{efinite3} and \eqref{efinite}, this upper bound tends
to zero as $r\to\infty$.

It remains to consider  the case where $k_r\equiv0$. Since $\BE T_r(0)=0$
and the centring term of $T_r(0)$ (in the case $\ell_r\equiv 0$) is deterministic
we then have
\begin{align*}
\BE T_r(0)\overline{T_r(\ell_r)}
&= \frac{1}{r^d} \BE\bigg[\sum_{y\in\eta}h_r(y)\sum_{x\in\eta}h_r(x)\te^{\ti\langle \ell_r,x\rangle}\bigg]
-M_1(h)\gamma \BE\bigg[\sum_{x\in\eta}h_r(x)\te^{\ti\langle \ell_r,x\rangle}\bigg]\\
&= \frac{\gamma}{r^d} \int h_r(x)^2\te^{\ti\langle \ell_r,x\rangle}\,\D x
+ \frac{1}{r^d} \iint h_r(x)h_r(x+y)\te^{\ti\langle \ell_r,x\rangle}\,\D x\,\alpha^!_2(\D y)\\
&\quad -M_1(h)\gamma^2 \int h_r(x)\te^{-\ti\langle \ell_r,x\rangle}\,\D x.
\end{align*}
Since $\int h_r(x+y)\,\D y=M_1(h)r^d$, $x\in\R^d$,
we obtain
\begin{align}\label{2078}
\BE T_r(0)\overline{T_r(\ell_r)}&= \frac{\gamma}{r^d} \int h_r(x)^2\te^{\ti\langle \ell_r,x\rangle}\,\D x\\ \notag
&\quad+ \frac{1}{r^d} \iint \te^{\ti\langle \ell_r,x\rangle}h_r(x)h_r(x+y)\,\D x\,\beta^!_2(\D y).
\end{align}
If $\ell_r\equiv 0$, the first term equals $\gamma M_2(h)$. Otherwise, it tends to zero by \eqref{e977}.
Similarly as before, we rewrite the second term as
\begin{align}\label{2079}
&\frac{1}{r^d}\beta^!_2(\R^d) \int \te^{\ti\langle \ell_r,x\rangle}h_r(x)^2\,\D x\\ \notag
&\quad+\frac{1}{r^d} \iint \te^{-\ti\langle \ell_r,x\rangle}h_r(x)(h_r(x+y)-h_r(x))\,\D x\,\beta^!_2(\D y).
\end{align}
Here, the second term tends to $0$ as $r\to\infty$, and so does the first term if $(\ell_r)\ne \mathbf{0}$.
If $(\ell_r)=\mathbf{0}$ the first term equals $M_2(h)\beta^!_2(\R^d)=M_2(h)\gamma(S(0)-1)$.
\end{proof}

\subsection{Proof of Theorem \ref{clt} (central limit theorem)} \label{appendix2}

We adapt the classical methods from \cite{Brillinger72} to our spatial setting..
For a finite non-empty set $M$, let $\Pi_M$ denote the set of all
partitions of $M$. For $n\in\N$ we abbreviate $\Pi_n:=\Pi_{[n]}$.
Let $m\in\N$, and let $X_1,\ldots,X_m$ be complex-valued random variables
satisfying $\BE |X_j|^m<\infty$ for each $j\in[m]$.
The (joint) {\em cumulant} of these random variables is defined
as
\begin{align}\label{ecumu}
\Cum(X_1,\ldots,X_m)=
\sum_{\sigma\in \Pi_m}(-1)^{|\sigma|-1}(|\sigma|-1)!\prod_{I\in\sigma}\BE \prod_{i\in I}X_i.
\end{align}
For real-valued random variables, \eqref{ecumu} is not the definition
but rather a fundamental property; see, e.g., \cite[Proposition 3.2.1]{PeccTaqqu11}.
However, following \cite{Brillinger72}, 
we use \eqref{ecumu} as a natural multilinear extension of the cumulant to
random vectors with complex-valued components.
The inverse relation of \eqref{ecumu} is given by
\begin{align}\label{ecumuinv}
\BE X_1\cdots X_m=\sum_{\sigma\in \Pi_m}\prod_{I\in\sigma}\Cum\big((X_j)_{j\in I}\big).
\end{align}

For any measurable function $f\colon\R^d\to\C$, we define
\begin{align*}
\eta(f):=\sum_{x\in\eta}f(x).
\end{align*}
Let $m\in\N$, and let $f_1,\ldots,f_m\colon\R^d\to\C$ be measurable
and bounded functions. For $I\subset[m]$, we define
$f_I\colon\R^d\to\C$ by $f_I(x):=\prod_{i\in I}f_i(x)$, $x\in\R^d$.
Using \eqref{ecumu} and following the argument presented in the proof of \cite[Lemma 1]{HeinrichSchmidt85},
one can show that
\begin{align}\label{ecumu2}
\Cum\big(\eta(f_1),\ldots,\eta(f_m)\big)
=\sum_{i=1}^m\sum_{\sigma\in\Pi_{m,i}}
\int \prod^i_{j=1}f_{I_j(\sigma)}(x_j)\, \beta_{i}(\D (x_1,\ldots,x_i)),
\end{align}
where we write $\sigma=\{I_1(\sigma),\ldots,I_i(\sigma)\}$ for $\sigma\in\Pi_{m,i}$
and assume that
\begin{align}\label{e:int}
\int \prod^i_{j=1}|f_{I_j(\sigma)}(x_j)|\, |\beta_{i}|(\D (x_1,\ldots,x_i))<\infty,\quad \sigma\in\Pi_{m,i},\,i\le m.
\end{align}
For the benefit of the reader, and since the argument applies to point processes
on a general state space, we outline the key idea behind the proof of \eqref{ecumu2}.
We first assume that $f_1,\ldots,f_m$ vanish outside a bounded set.
It is not difficult to see (see \cite[Exercise 5.4.5]{DaleyVereJones}) that, under this condition, \begin{align*}
\BE \prod_{i\in I}\eta(f_i)=
\sum_{\pi\in \Pi_I}\int\prod_{J\in\pi}f_J(x_J)\,\alpha_{|\pi|}(\D(x_J)_{J\in\pi}).
\end{align*}
Substituting this expression into \eqref{ecumu} yields
\begin{align*}
\Cum\big(\eta(f_1),\ldots,\eta(f_m)\big)&=
\sum_{\sigma\in \Pi_m}(-1)^{|\sigma|-1}(|\sigma|-1)!\\
&\qquad \qquad\times \prod_{I\in\sigma} \sum_{\pi\in \Pi_I}\int\prod_{J\in\pi}f_J(x_J)\,\alpha_{|\pi|}(\D(x_J)_{J\in\pi}).
\end{align*}
We now express the product over $I\in\sigma$ as a sum over
$\pi_1\in \Pi_{I_1(\sigma)},\ldots,\pi_{|\sigma|}\in \Pi_{I_{|\sigma|}(\sigma)}$
and then swap the summation with that over $\sigma$. This reordering results
in a sum over $\pi\in\Pi_m$, followed by a sum
over $\sigma\in\Pi_\pi$. Finally, using definition
\eqref{deffaccum} with $\pi$ instead of $[m]$, we obtain \eqref{ecumu2}.

If some of the $f_1,\ldots,f_m$ do not have a bounded support, we can apply
\eqref{ecumu2} to their restrictions $f^r_1,\ldots,f^r_m$ to the ball $B_r$, $r>0$. By assumption 
\eqref{e:int} and dominated convergence, the corresponding right-hand side of  \eqref{ecumu2} 
converges, as $r\to \infty$, to the right-hand side of  \eqref{ecumu2}.
This remains true if we replace the functions $f_1,\ldots,f_m$ by their absolute values.
But then we can use \eqref{ecumuinv} and monotone convergence too see that
$\BE \eta(|f_j|)^m<\infty$ for each $j\in[m]$. Finally it follows from definition \eqref{ecumu}
and dominated convergence that 
\begin{align*}
\lim_{r\to\infty}\Cum\big(\eta(f^r_1),\ldots,\eta(f^r_m)\big)=\Cum\big(\eta(f_1),\ldots,\eta(f_m)),
\end{align*}
concluding the proof of \eqref{ecumu2} in the general case.

We now consider $m\ge 3$, $r>0$, and $k_1(r),\ldots,k_m(r)\in\R^d$. 
Since the dependence of $k_i(r)$ on $r$ is not essential in the following discussion, we omit it from our notation. 
(In view of \eqref{eclt}, this constitutes  a slight abuse of notation.) 
Applying equation 
\eqref{ecumu2} with the functions $f_j:=h_r\te^{-\ti\langle k_j,\cdot\rangle}$, $j\in[m]$,
the right-hand side of \eqref{ecumu2}, denoted by $A(r)$, simplifies to
\begin{align*}
A(r)=\sum_{i=1}^m\sum_{\sigma\in\Pi_{m,i}}\int \prod^i_{q=1}
\prod_{j\in I_q(\sigma)}h_r(x_q)\te^{-\ti\langle k_j,x_q\rangle}
\, \beta_{i}(\D (x_1,\ldots,x_i)).
\end{align*}
Replacing the integrands by their absolute values and the measures
$\beta_{i}$ by $|\beta_{i}|$, we can use
the integrability of $h$ and the identity
\begin{align*}
|\beta_m|(\cdot)=
\iint\I\{(x,x_1+x,\ldots,x_{m-1}+x)\in \cdot\}\,|\beta^!_m|(\D (x_1,\ldots,x_{m-1}))\,\D x
\end{align*}
to confirm assumption \eqref{e:int}.
Defining $k':=k_1+\cdots+k_m$ and using \eqref{ereducedcum}, we obtain
\begin{align*}
A(r)&=\sum_{i=1}^m\sum_{\sigma\in\Pi_{m,i}}\iint \prod_{j\in I_1(\sigma)}h_r(x)\te^{-\ti\langle k_j,x\rangle} \\
&\qquad \qquad\times\prod^i_{q=2}
\prod_{j\in I_q(\sigma)}h_r(x+x_q)\te^{-\ti\langle k_j,x_q\rangle}\te^{-\ti\langle k_j,x\rangle}
\, \beta^!_i(\D (x_2,\ldots,x_{i}))\,\D x\\
&=\sum_{i=1}^m\sum_{\sigma\in\Pi_{m,i}}\iint \te^{-\ti\langle k',x\rangle}h_r(x)^{|I_1(\sigma)|}\\
&\qquad \qquad\times\prod^i_{q=2}\prod_{j\in I_q(\sigma)}h_r(x+x_q)\te^{-\ti\langle k_j,x_q\rangle}
\, \beta^!_i(\D (x_2,\ldots,x_{i}))\,\D x,
\end{align*}
where the term for $i=1$ should be interpreted appropriately.
Now, suppose $a_1,b_1,\ldots,a_m$, $b_m\in\R$ are bounded in absolute value by some constant $c>1$. It then follows by
induction (using the identity $ab-cd=(a-c)b+(b-d)$) 
\begin{align*}
\Big|\prod^m_{j=1}a_j-\prod^m_{j=1}b_j\Big|
\le  c^{m-1} \sum^m_{j=1}|a_j-b_j|. 
\end{align*}
Letting $c$ be an upper bound of  $|h|$, and defining $c_1:= mc^{m-1}$, we  obtain,
for each $\sigma\in\Pi_{m,i}$, that
\begin{align*}
\Big|h_r(x)^{|I_1(\sigma)|}\prod^i_{j=2}&h_r(x_j+x)^{|I_j(\sigma)|}-h_r(x)^m\Big|\\
&\le m c^{m-1}\sum^i_{j=2}|I_j(\sigma)||h_r(x_j+x)-h_r(x)|\\
&\le c_1\sum^i_{j=2}|h_r(x_j+x)-h_r(x)|.
\end{align*}
It follows that $A(r)=A_1(r)+A_2(r)$,
where
\begin{align*}
A_1(r)&:=\sum_{i=1}^m\sum_{\sigma\in\Pi_{m,i}}\iint \te^{-\ti\langle k',x\rangle}h_r(x)^m\\
&\qquad \qquad\times\prod^i_{j=2}\prod_{i\in I_j(\sigma)}\te^{-\ti\langle k_i,x_j\rangle}
\, \beta^!_n(\D (x_2,\ldots,x_i))\,\D x,
\end{align*}
and $A_2(r)$ satisfies
\begin{align}\label{e714}\notag
|A_2(r)|&\le c_2 \sum^m_{i=2}\sum^i_{j=2} \iint |h_r(x_j+x)-h_r(x)|\,\D x\,|\beta^!_i|(\D (x_2,\ldots,x_i))\\
&= c_2 r^d\sum^m_{i=2}\sum^i_{j=2} \iint |h(r^{-1}x_j+x)-h(x)|\,\D x\,|\beta^!_i|(\D (x_2,\ldots,x_i))
\end{align}
for some $c_2>0$.
By assumption \eqref{Bmixing} and the fact that $\int h(x)^m\,\D x<\infty$, we conclude that
\begin{align}\label{e407}
|A_1(r)|\le c_3r^d,\quad r>0,
\end{align}
for some constant $c_3>0$. Additionally, since $\int |h(z+x)-h(x)|\,\D x\le 2M_1(h)$
for all $z\in\R^d$, the same bound applies to $|A_2(r)|$.
If $k_j\ne 0$ then $T_r(k_j)=r^{-d/2}\eta(f_j)$. Since
the cumulant is homogeneous in each variable,
it follows that
\begin{align}\label{e456}
  \lim_{r\to\infty} \Cum\big(T_r(k_1(r)),\ldots,T_r(k_m(r))\big)=0,\quad m\ge 3.
\end{align}
This remains true if one of the $k_j$ equals $0$, because
the cumulant is translation invariant.

Recall that $T_{r,1}(k)$ denotes the real part and $T_{r,2}(k)$
the imaginary part of $T_r(k)$.
Let us abbreviate $(X'_j,Y'_j,Z'_j):=(T_{r,1}(k_j),T_{r,2}(k_j),T_r(k_j))$. 
By the multilinearity of cumulants, 
$\Cum(X'_1,Z'_2,\ldots,Z'_m)$ and $\Cum(Y'_1,Z'_2,\ldots,Z'_m)$ are linear
functions of $\Cum(Z'_1,\ldots,Z_m')$ and $\Cum(\overline{Z'_1},Z_2,\ldots,Z_m)$.
Let $W_j\in\{X_j',Y_j'\}$ for $j\in[m]$.
By induction,  it follows that $\Cum(W_1,\ldots,W_m)$
is a linear function of cumulants
$\Cum(V_1,\ldots,V_m)$, where $V_j\in\{Z'_j,\overline{Z'_j}\}$.
Since we may apply \eqref{e456} with $k'_j$ in place of $k_j$,
where $k'_j\in\{k_1,-k_1,\ldots,k_m,-k_m\}$,
and because $T_r(-k_j)=\overline{T_r(k_j)}$, we obtain
\begin{align}\label{e555}
\lim_{r\to\infty} \Cum\big(Z_r(k_1(r)),\ldots,Z_r(k_m(r))\big)=0,
\quad k_1,\ldots,k_m\in\R^d,\,m\ge 3,
\end{align}
where $Z_r(k_j)\in\{T_{r,1}(k_\ell),T_{r,2}(k_\ell):\ell\in[m]\}$ for each $j\in[m]$.

With these preliminaries in place, we now turn to the proof of the convergence \eqref{eclt}.
Let
$a_1,b_1\ldots,a_n,b_n$ be real  numbers. By the Cram\'{e}r--Wold device, it suffices to show that
\begin{align*}
Z_r:=\sum^n_{j=1} \big(a_jT_{r,1}(k_{j,r})+b_jT_{r,2}(k_{j,r})\big)
\overset{d}{\longrightarrow}\sum^{n}_{j=1}\big(a_jX_j+b_jY_j\big),\quad \text{as $r\to\infty$},
\end{align*}
where $X_1,Y_1,\ldots,X_n,Y_n$ are independent, normally
distributed random variables with mean $0$ and variances
$\gamma M_2(h)S(k_1)/2,\ldots, \gamma M_2(h)S(k_n)/2$, respectively.  To prove this, we use the method of moments 
(see, e.g., \cite[Section 30]{Billingsley}, or \cite[Section 3]{henze24}).
By equations \eqref{ecumu} and \eqref{ecumuinv},
cumulants and moments are in a one-to-one (polynomial) correspondence,
so  convergence in distribution can be established by showing the convergence of the cumulants.
Specifically, for $m\in\N$,
the $m$-th cumulant of $Z_r$ is defined by
$C_m(Z_r):=\Cum(Z_r,\ldots,Z_r)$, where $Z_r$ appears $m$ times. From \eqref{e555} and multilinearity, it follows  that
\begin{align*}
\lim_{r\to\infty} C_m(Z_r)=0,\quad m\ge 3.
\end{align*}
From \eqref{e3.45} and \eqref{e590}, we further conclude that 
\begin{align*}
\lim_{r\to\infty}\BV(Z_r)=\frac{\gamma}{2}\sum^n_{j=1}(a_j^2+b^2_j)M_2(h)S(k_j),
\end{align*}
which matches the variance of $\sum^{n}_{j=1}(a_jX_j+b_jY_j)$.
It only  remains to invoke once again \eqref{e590} to ensure convergence  of the first moments.

\begin{remark}\rm It is worth noting that the proof of the convergence
\eqref{e456} requires only that $h$ be bounded and integrable,
along with the hypothesis \eqref{Bmixing} regarding the reduced cumulant measures of order $m\ge 3$. 
The conditions 
\eqref{e234} and \eqref{e239} are not needed.
Hence, the CLT \eqref{eclt}  remains valid in a more general setting, 
provided that the asymptotic expectations \eqref{e590} and  covariances \eqref{e:ascov} exist as finite limits.
\end{remark}

\subsection{Properties of the likelihood function} \label{appendix3}

Here we discuss some properties of the likelihood function $\mathcal{L}^*$ defined by 
\eqref{e:L*}. Using \eqref{e:L*2} and 
some straightforward calculations, we obtain
\begin{align}\label{stablell}\notag
\mathcal{L}^*(\vartheta)=
&- \sum_{j=1}^n \log\big(\cos\vartheta + \kappa_j \sin\vartheta \big)  - 
n \log \left(\sum_{j=1}^n \frac{x_j}{\cos\vartheta+\kappa_j \sin\vartheta } \right)\\ 
&\quad + n(\log n - 1).
\end{align}
The derivative of $\mathcal{L}^*$ is given by
\begin{align}\label{lstarderivative}
\frac{\mathrm{d}}{\mathrm{d}\vartheta}\mathcal{L}^*(\vartheta) =
\sum_{j=1}^n \frac{\sin \vartheta - \kappa_j \cos \vartheta}{\cos \vartheta + \kappa_j \sin \vartheta} + n \cdot \frac{\displaystyle{\sum_{j=1}^n} \displaystyle{\frac{x_j(-\sin \vartheta + \kappa_j \cos \vartheta)}{(\cos \vartheta + \kappa_j \sin \vartheta)^2}}}{\displaystyle{\sum_{j=1}^n} \displaystyle{\frac{x_j}{\cos \vartheta + \kappa_j \sin \vartheta}} }.
\end{align}

Regarding the limits of ${\cal L}^*(\vartheta)$ as $\vartheta \uparrow \pi/2$ and $\vartheta \downarrow \vartheta_0$, we have the following result: 

\begin{lemma}\label{lemma1lq}
\begin{enumerate}
\item[a)] $\displaystyle{ \lim_{\vartheta \uparrow \pi/2} {\mathcal{L}^*}(\vartheta)  = - \sum_{j=1}^n \log \kappa_j - n \log 
\left(\sum_{j=1}^n \frac{x_j}{\kappa_j}\right)} + n(\log n -1)$,\\
\item[b)] If $\kappa_1 = \ldots = \kappa_n$, then
$\displaystyle{\lim_{\vartheta \downarrow \vartheta_0} {\mathcal{L}^*}(\vartheta)  =  - n \log \left(\sum_{j=1}^n x_j\right) + n(\log n -1)}$,
\item[c)] If at least two of the $\kappa_j$ differ from each other, then
$\displaystyle{\lim_{\vartheta \downarrow \vartheta_0} {\mathcal{L}^*}(\vartheta) =    - \infty}$.
\end{enumerate}
\end{lemma}
\begin{proof}  Part a) follows readily from \eqref{stablell}, since 
$\cos \frac{\pi}{2} =0$ and $\sin \frac{\pi}{2} =1$.  To prove b), let
$\kappa := \kappa_1 = \ldots = \kappa_n$. Using $\log \frac{a}{b} = \log a - \log b$, we obtain 
\begin{align*}
{\mathcal{L}^*}(\vartheta)  & 
=n \log(\cos \vartheta + \kappa \sin \vartheta) - n \log \left( \sum_{j=1}^n \frac{x_j}{\cos \vartheta + \kappa \sin \vartheta}\right) + n (\log n - 1)\\
& =- n \log \left( \sum_{j=1}^n x_j\right) + n\log n -1),
\end{align*}
proving the assertion.
To show c), assume without loss of generality that the $\kappa_j$ are ordered so that  $\kappa_1 \le \ldots \le \kappa_n$. Moreover, suppose
$\kappa_1\le  \ldots \le  \kappa_\ell <\kappa_{\ell +1} = \ldots = \kappa_n$ for some $\ell \in \{1,\ldots,n-1\}$. Setting 
$t:= 1+ \kappa_n \tan \vartheta$, the limit
$\vartheta \downarrow \vartheta_0$ corresponds to $t \downarrow 0$. Starting from \eqref{stablell} and using the  definitions of $\ell$ and $t$, we obtain
\begin{align*}
 {\mathcal L}^*(\vartheta) & =   - \sum_{j=1}^\ell \log(\cos \vartheta + \kappa_j \sin \vartheta) - (n-\ell) \log (t \cos \vartheta)\\
 &\quad - n \log\left(\frac{1}{t} \bigg\{ t \sum_{j=1}^\ell \frac{x_j}{\cos \vartheta + \kappa_j \sin \vartheta} + \frac{1}{\cos \vartheta} \sum_{j=\ell +1}^n x_j \bigg\}  \right)
+ n(\log n -1).
\end{align*} 
Writing $O(1)$ for a term that is bounded as $t \downarrow 0$, and using the identity $\log(ab) = \log a + \log b$, it follows that 
\[
{\mathcal L}^*(\vartheta) = \ell \log t + O(1) \quad \text{as} \ t \downarrow 0.
\]
Since $\ell \ge 1$, the proof is completed.
\end{proof}
   
Although we do not have further mathematical results on the behavior of
the function $\mathcal{L}^*$, there is strong heuristic evidence that
$\mathcal{L}^*$ has a unique maximizer $\widehat{\vartheta} \in
(\vartheta_0, \pi/2]$, at least for the classes of problems that
arise in typical applications. 
Moreover, this maximizer is equal to $\pi /2$ if the left-hand derivative of
$\mathcal{L}^*$ at $\pi /2$} is positive, and negative
if this left-hand derivative is positive. 
From \eqref{lstarderivative}, it follows that 
\begin{align*}
  ({\mathcal L}^*)'\left( \frac{\pi}{2}-\right) := 
\lim_{\vartheta    \uparrow \pi /2}
  \frac{\mathrm{d}}{\mathrm{d}\vartheta}\mathcal{L}^*(\vartheta) =
  \sum_{j=1}^n \frac{1}{\kappa_j} - n \cdot
  \frac{\displaystyle{\sum_{j=1}^n} \displaystyle{
      \frac{x_j}{\kappa_j^2}}}{\displaystyle{\sum_{j=1}^n}
    \displaystyle{ \frac{x_j}{\kappa_j}}},
\end{align*}
which implies
\begin{align*}
({\cal L}^*)'\left( \frac{\pi}{2}-\right) > 0  \Longleftrightarrow 
\sum_{j=1}^n \frac{x_j}{\kappa_j} \left(\frac{1}{n}\sum_{i=1}^n \frac{1}{\kappa_i} - \frac{1}{\kappa_j} \right) > 0.
\end{align*}

\subsection*{Acknowledgments}

We thank two anonymous referees, the authors of \cite{HGBL22}, Philipp
Neumann, D. Yogeshwaran and Joe Yukich for many helpful comments.
This research was supported by the Initiative and Networking Fund of the
Helmholtz Association through the Project ``DataMat'', as well as by the
Deutsche Forschungsgemeinschaft (DFG, German Research Foundation) as
part of the DFG priority program `Random Geometric Systems' SPP 2265)
under grants KL 3391/2-2, LA 965/11-1/2, LO 418/25-1, ME 1361/16-1, and
WI 5527/1-1, and by the Volkswagenstiftung via the `Experiment' Project
`Finite Projective Geometry'.

\end{document}